\documentclass[reqno,11pt,a4paper]{article}

\usepackage{amsmath,amssymb,amsthm}
\usepackage[UKenglish]{babel}
\usepackage{bbm}                                 
\usepackage[right=2.54cm,left=2.54cm]{geometry}
\usepackage{ifthen}                              
\usepackage{multicol}                            
\usepackage{sectsty}
\usepackage[pdftitle={An algebraic construction of quantum flows with unbounded generators},
            pdfauthor={Alexander C. R. Belton and Stephen J. Wills},
            bookmarks=true,
            bookmarksnumbered=true,
            colorlinks=false,
            pdfstartview=FitH]{hyperref}

\sectionfont{\large}
\subsectionfont{\normalsize}

\mathchardef\ordinarycolon\mathcode`\:
\mathcode`\:=\string"8000
\def\vcentcolon{\mathrel{\mathop\ordinarycolon}}
\begingroup \catcode`\:=\active
\lowercase{\endgroup
\let :\vcentcolon}

\theoremstyle{plain}
\newtheorem{theorem}{Theorem}[section]
\newtheorem{lemma}[theorem]{Lemma}
\newtheorem{proposition}[theorem]{Proposition}
\newtheorem{corollary}[theorem]{Corollary}

\theoremstyle{definition}
\newtheorem{definition}[theorem]{Definition}
\newtheorem{example}[theorem]{Example}
\newtheorem{notation}[theorem]{Notation}
\newtheorem{remark}[theorem]{Remark}

\let\origthebibliography=\thebibliography
\def\thebibliography{\renewcommand{\section}[2]{}\origthebibliography}

\newcommand{\alg}{\mathcal{A}}
\newcommand{\Aut}{\mathop{\mathrm{Aut}}}
\newcommand{\blg}{\mathcal{B}}
\newcommand{\bop}[2]%
{\ifthenelse{\equal{#2}{}}{\bopp(#1)}{\bopp(#1;#2)}}
\newcommand{\bopp}{\mathcal{B}}
\newcommand{\bra}[1]{\langle#1|}
\newcommand{\bt}{\mathbf{t}}
\newcommand{\cb}{\text{cb}}
\newcommand{\comp}{\mathbin{\circ}}
\newcommand{\ctimes}{\mathbin{\bar{\otimes}}}
\newcommand{\dom}{\mathop{\mathrm{dom}}}
\newcommand{\dyad}[2]{|#1\rangle\langle#2|}
\newcommand{\elltwo}{L^2( \R_+; \mul )}
\newcommand{\evec}[1]{\varepsilon(#1)}
\newcommand{\evecs}{\mathcal{E}}
\newcommand{\expn}{\mathbb{E}}
\newcommand{\fock}{\mathcal{F}}
\newcommand{\hilb}{\mathsf{H}}
\newcommand{\hlf}{\mbox{$\frac12$}}
\newcommand{\id}{1}
\newcommand{\ini}{\mathsf{h}}
\newcommand{\jt}{\bar{\jmath}}
\newcommand{\ket}[1]{|#1\rangle}
\newcommand{\lin}{\mathop{\mathrm{lin}}}
\newcommand{\matten}{\mathbin{\otimes_{\mathrm{M}}}}
\newcommand{\mul}{\mathsf{k}}
\newcommand{\mmul}{{\wh{\mul}}}
\renewcommand{\Pr}{\mathbb{P}}
\newcommand{\rd}{\mathrm{d}}
\newcommand{\simp}{D}
\newcommand{\supp}{\mathop{\mathrm{supp}}\nolimits}
\newcommand{\std}{\,\rd}
\newcommand{\tfn}[1]{\mathbbm{1}_{#1}}
\newcommand{\tr}{\mathop{\mathrm{tr}}\nolimits}
\newcommand{\vac}{\omega}
\newcommand{\wh}[1]{\widehat{#1}}
\newcommand{\wt}[1]{\widetilde{#1}}

\newcommand{\C}{\mathbb{C}}
\newcommand{\I}{\mathrm{i}}
\newcommand{\N}{\mathbb{N}}
\newcommand{\R}{\mathbb{R}}
\newcommand{\T}{\mathbb{T}}
\newcommand{\Z}{\mathbb{Z}}

\renewcommand{\ge}{\geqslant}
\renewcommand{\le}{\leqslant}

\newcommand{\etc}{\textit{et cetera}}
\newcommand{\eg}{\textit{e.g.},\ }
\newcommand{\ie}{\textit{i.e.},\ }

\newcommand{\ds}{\displaystyle}

\numberwithin{equation}{section}

\begin{document}

\begin{center}
{\LARGE An algebraic construction of quantum\\[0.25ex]
flows with unbounded generators}
\vspace*{1ex}
\begin{multicols}{2}
{\large Alexander C.~R.~Belton}\\[0.5ex]
{\small Department of Mathematics and Statistics\\
Lancaster University, United Kingdom\\[0.5ex]
\textsf{a.belton@lancaster.ac.uk}}
\columnbreak

{\large Stephen J.~Wills}\\[0.5ex]
{\small School of Mathematical Sciences\\
University College Cork, Ireland\\[0.5ex]
\textsf{s.wills@ucc.ie}}
\end{multicols}
{\small \today}
\end{center}

\begin{abstract}
{\small \noindent
It is shown how to construct $*$-homomorphic quantum stochastic Feller
cocycles for certain unbounded generators, and so obtain dilations of
strongly continuous quantum dynamical semigroups on $C^*$~algebras;
this generalises the construction of a classical Feller process and
semigroup from a given generator. The construction is possible
provided the generator satisfies an invariance property for some dense
subalgebra $\alg_0$ of the $C^*$~algebra $\alg$ and obeys the
necessary structure relations; the iterates of the generator, when
applied to a generating set for $\alg_0$, must satisfy a growth
condition. Furthermore, it is assumed that either the
subalgebra~$\alg_0$ is generated by isometries and $\alg$ is
universal, or $\alg_0$ contains its square roots. These conditions are
verified in four cases: classical random walks on discrete groups,
Rebolledo's symmetric quantum exclusion processes and flows on the
non-commutative torus and the universal rotation algebra.}
\end{abstract}

{\small\textit{Key words:} quantum dynamical semigroup; quantum Markov
semigroup; CPC semigroup; strongly continuous semigroup; semigroup
dilation; Feller cocycle; higher-order It\^{o} product formula; random
walks on discrete groups; quantum exclusion process; non-commutative
torus}

{\small\textit{MSC 2000:}
81S25 (primary);    
46L53,              
46N50,              
47D06,              
60J27 (secondary).} 

\section{Introduction}

The connexion between time-homogeneous Markov processes and
one-parameter contraction semigroups is an excellent example of the
interplay between probability theory and functional analysis. Given a
measurable space $( E, \mathcal{E} )$, a \emph{Markov semigroup} $T$
with state space $E$ is a family $( T_t )_{t \ge 0}$ of positive
contraction operators on $L^\infty( E )$ such that
\[
T_{s + t} = T_s \comp T_t \quad \text{for all } s, t \ge 0 %
\qquad \text{and} \qquad T_0 f = f %
\quad \text{for all } f \in L^\infty( E );
\]
the semigroup is \emph{conservative} if $T_t 1 = 1$ for all $t \ge 0$.
Typically, such a semigroup is defined by setting
\[
( T_t f )( x ) = \int_E f( y ) p_t( x, \rd y )
\]
for a family of transition kernels
$p_t : E \times \mathcal{E} \to [ 0, 1 ]$. Given a time-homogeneous
Markov process $( X_t )_{t \ge 0}$ with values in $E$, the associated
Markov semigroup is obtained from the prescription
\begin{equation}\label{eqn:markov}
( T_t f )( x ) = \expn[ f( X_t ) | X_0 = x ],
\end{equation}
so that $p_t( x, A ) = \mathbb{P}( X_t \in A | X_0 = x )$ is the
probability of moving from $x$ into $A$ in time~$t$. When the state
space~$E$ is a locally compact Hausdorff space we may specialise
further: a \emph{Feller semigroup} is a Markov semigroup $T$ such that
\[
T_t\bigl( C_0 ( E ) \bigr) \subseteq C_0 ( E ) %
\quad \text{for all } t \ge 0 \quad \text{ and } \quad %
\| T_t f - f \|_\infty \to 0 \text{ as } t \to 0 %
\quad \text{ for all } f \in C_0 ( E ).
\]
Any sufficiently nice Markov process, such as a L\'{e}vy process,
gives rise to a Feller semigroup; conversely, if $E$ is separable then
any Feller semigroup gives rise to a Markov process with
\textit{c\`{a}dl\`{a}g} paths.

A celebrated theorem of Gelfand and Naimark states that every
commutative $C^*$~algebra is of the form $C_0( E )$ for some locally
compact Hausdorff space $E$. Thus the first step in generalising
Feller semigroups, and so Markov processes, to a non-commutative
setting is to replace $C_0( E )$ with a general
$C^*$~algebra~$\alg$. Moreover, a strengthening of positivity, called
\emph{complete positivity}, is required for a satisfactory theory: a
map $\phi : \alg \to \blg$ between $C^*$~algebras is completely
positive if the ampliation
\[
\phi^{(n)} : M_n( \alg ) \to M_n( \blg ); \ %
( x_{i j} ) \mapsto \bigl( \phi( x_{i j} ) \bigr)
\]
is positive for all $n \ge 1$. This property is justified on physical
grounds and is equivalent to the usual form of positivity when either
algebra $\alg$ or $\blg$ is commutative. The resulting object, a
semigroup of completely positive contractions on a $C^*$~algebra
$\alg$, is known as a \emph{quantum dynamical semigroup} or, when
conservative, a \emph{quantum Markov semigroup}; such semigroups are
used to describe the evolution of quantum-mechanical systems which
interact irreversibly with their environment.

Any strongly continuous quantum dynamical semigroup $T$ is
characterised by its infinitesimal generator $\tau$, the closed
linear operator such that
\[
\dom \tau = \Bigl\{ f \in \alg : %
\lim_{t \to 0} \frac{T_t f - f}{t} \text{ exists} \Bigr\} \quad %
\text{and} \quad \tau f = \lim_{t \to 0} \frac{T_t f - f}{t}.
\]
For a Feller semigroup, the form of the generator $\tau$ may reveal
properties of the corresponding process; for instance, a classical
L\'{e}vy process may be specified, \textit{via} the
L\'{e}vy--Khintchine formula, by the characteristics of its generator,
\textit{viz.}\ a drift vector, a diffusion matrix describing the
Brownian-motion component and a L\'{e}vy measure characterising its
jumps. If we start with a putative generator $\tau$ then
operator-theoretic methods may be used to construct the semigroup,
although there are often considerable analytical challenges to be
met. Verifying that $\tau$ satisfies the hypotheses of the
Hille--Yosida theorem, the key analytical tool for this construction,
is often difficult. In this paper we provide, for a suitable class of
generators, another method of constructing quantum dynamical
semigroups and the corresponding non-commutative Markov processes.

To understand how the relationship between semigroups and Markov
processes generalises to the non-commutative framework, recall first
that any locally compact Hausdorff space $E$ may be made compact by
adjoining a point at infinity, which corresponds to adding an identity
to the algebra $C_0( E )$ or adding a coffin state for an $E$-valued
Markov process; it is sufficient, therefore, to restrict our attention
to compact Hausdorff spaces or, equivalently, unital
$C^*$~algebras. The correct analogue of an $E$-valued random variable
$X$ is then a unital $*$-homomorphism $j$ from $\alg$ to some unital
$C^*$~algebra $\blg$; classically, $j$ is the map $f \mapsto f \comp
X$, where $f \in \alg = C_0( E )$ and~$\blg$ is~$L^\infty( \Pr )$ for
some probability measure $\Pr$. A family of unital $*$-homomorphisms
$( j_t : \alg \to \blg )_{t \ge 0}$, \ie a non-commutative stochastic
process, is said to \emph{dilate} the quantum dynamical semigroup~$T$
if $\alg$ is a subalgebra of $\blg$ and $\expn \comp j_t = T_t$ for
all $t \ge 0$, where~$\expn$ is a conditional expectation from $\blg$
to $\alg$; the relationship to \eqref{eqn:markov} is clear. Thus
finding a dilation for a given semigroup is analogous to constructing
a Markov process from a family of transition kernels.

The tool used here for constructing semigroups and their dilations is
a stochastic calculus: the quantum stochastic calculus introduced by
Hudson and Parthasarathy in their 1984 paper \cite{HuP84}. In its
simplest form, this is a non-commutative theory of stochastic
integration with respect to three operator martingales which
correspond to the creation, annihilation and gauge processes of
quantum field theory. It generalises simultaneously the It\^{o}--Doob
$L^2$ integral with respect to either Brownian motion or the
compensated Poisson process; as emphasised by Meyer \cite{Mey95} and
Attal \cite{Att98}, the $L^2$~theory of any normal martingale having
the chaotic-representation property, such as Brownian motion, the
compensated Poisson process or Az\'{e}ma's martingale, gives a
classical probabilistic interpretation of Boson Fock space, the
ambient space of quantum stochastic calculus.

We develop below new techniques for obtaining $*$-homomorphic
solutions to the Evans--Hudson quantum stochastic differential
equation (QSDE)
\begin{equation}\label{eqn:EHqsde}
\rd j_t = %
( j_t \otimes \iota_{\bop{\mmul}{}} ) \comp \phi \std \Lambda_t,
\end{equation}
where the solution $j_t$ acts on a unital $C^*$~algebra $\alg$.
In this way, we obtain the process $j$ and the quantum dynamical
semigroup $T$ simultaneously. The components of the \emph{flow
generator}~$\phi$ include $\tau$, the restriction of a semigroup
generator, and $\delta$, a bimodule derivation, which are related to
one another through the Bakry--\'{E}mery \textit{carr\'{e} du champ}
operator: see Remark~\ref{rem:BEcdc}. If~$\alg$ is commutative then,
by Theorem~\ref{thm:abelian}, the process $j$ is classical, in the
sense that the algebra generated by
$\{ j_t( a ): t \ge 0, \ a \in \alg \}$ is also commutative.

The recent expository paper \cite{Bia10} on quantum stochastic
methods, written for an audience of probabilists, includes
Parthasarathy and Sinha's method \cite{PaS90} for constructing
continuous-time Markov chains with finite state spaces by solving
quantum stochastic differential equations. To quote Biane,
\begin{quote}
``It may seem strange to the classical probabilist to use
noncommutative objects in order to describe a perfectly commutative
situation, however, this seems to be necessary if one wants to deal
with processes with jumps \ldots\ The right mathematical notion
\ldots, which generalizes to the noncommutative situation, is that of
a derivation into a bimodule \ldots\ Using this formalism, we can use
the Fock space as a uniform source of noise, and construct general
Markov processes (both continuous and discontinuous) using stochastic
differential equations.''.
\end{quote}
The results herein give a further illustration of this philosophy.

The use of quantum stochastic calculus to produce dilations has now
been studied for nearly thirty years. Most results, by Hudson and
Parthasarathy, Fagnola, Mohari, Sinha \etc, are obtained in the case
that $\alg = \bop{\ini}{}$ by first solving an operator-valued QSDE,
the Hudson--Parthasarathy equation, to obtain a unitary process~$U$,
and defining~$j$ through conjugation by $U$; see~\cite{Fag99} and
references therein. The corresponding theory for the Heisenberg
rather than the Schr\"{o}dinger viewpoint, solving the Evans--Hudson
equation~\eqref{eqn:EHqsde}, has mainly been developed under the
standing assumption that the generator $\phi$ is completely bounded,
which is necessary if the corresponding semigroup~$T$ is norm
continuous \cite{LiW00a}. When one deviates from this assumption,
which is analytically convenient but very restrictive, there are few
results. The earliest general method is due to Fagnola and Sinha
\cite{FaS93}, with later results by Goswami, Sahu and Sinha for a
particular model \cite{GSS05} and a more general method developed by
Goswami and Sinha in \cite{SiG07}. Another approach based on semigroup
methods has yet to yield existence results for the Evans--Hudson
equation: see \cite{AcK01} and \cite{LiW12}.

Our method here has an attractive simplicity, imposing minimal
conditions on the generator $\phi$. It must be a $*$-linear map
\[
\phi : \alg_0 \to \alg_0 \otimes \bop{\C \oplus \mul}{},
\]
where $\alg_0$ is a dense $*$-subalgebra of the unital $C^*$~algebra
$\alg \subseteq \bop{\ini}{}$ which contains $\id = \id_\ini$
and $\mul$ is a Hilbert space, called the \emph{multiplicity space},
the dimension of which measures the amount of noise available in the
system. This incorporates an assumption that, if $\phi$ is viewed as a
matrix of maps, its components leave~$\alg_0$ invariant, a hypothesis
also used in~\cite{FaS93}. Furthermore, $\phi$ must be such that
$\phi( \id ) = 0$ and the \emph{first-order It\^o product formula}
holds:
\begin{equation}\label{eqn:firstIto}
\phi( x y ) = \phi( x ) ( y \otimes \id_\mmul ) + %
( x \otimes \id_\mmul ) \phi( y ) + %
\phi( x ) \Delta \phi( y )
\qquad \text{for all } x, y \in \alg_0,
\end{equation}
where $\mmul := \C \oplus \mul$ and
$\Delta \in \alg_0 \otimes \bop{\mmul}{}$ is the orthogonal projection
onto $\ini \ctimes \mul$. Both these conditions are known to be
necessary if $\phi$ is to generate a family of unital
$*$-homomorphisms. Finally, a growth bound must be established for the
iterates of $\phi$ applied to elements taken from a suitable subset of
$\alg_0$.

Our approach is an elementary one for those adept in quantum
stochastic calculus, relying on familiar techniques such as
representing the solution to the Evans--Hudson QSDE as a sum of
quantum Wiener integrals. An essential tool is the higher-order It\^o
product formula, presented in Section~\ref{sec:hopf}. This formula was
first stated, for finite-dimensional noise, in \cite{CEH95}, was
proved for that case in \cite{HPu98} and reached its definitive form
in~\cite{LiW03}. In that last paper it was shown
that~\eqref{eqn:firstIto} is but the first of a sequence of identities
that must be satisfied in order to show that the solution $j$ of the
QSDE is weakly multiplicative. However, there are many situations in
which the validity of~\eqref{eqn:firstIto} implies that the other
identities hold \cite[Corollary~4.2]{LiW03}, and this is the case for
$\phi$ as above. Moreover, one of our main observations,
Corollary~\ref{cor:bound}, is that, by exploiting the algebraic
structure imposed by this sequence of identities, it is sufficient to
establish pointwise growth bounds on a $*$-generating set of~$\alg_0$;
this is a major simplification when compared with~\cite{FaS93}. Also,
by using the coordinate-free approach to quantum stochastic analysis
given in~\cite{Lin05}, we can take $\mul$ to be any Hilbert space,
removing the restriction in~\cite{FaS93} that $\mul$ be finite
dimensional.

The growth bounds obtained Section~\ref{sec:hopf} are employed in
Section~\ref{sec:qf} to produce a family of weakly multiplicative
$*$-linear maps from the algebra $\alg_0$ into the space of linear
operators in~$\ini \ctimes \fock$, where $\fock$ is the Boson Fock
space over $\elltwo$. It is shown that these maps extend to unital
$*$-homomorphisms in two distinct situations.
Theorem~\ref{thm:extension1}, which includes the case of AF~algebras,
exploits a square-root trick that is well known in the literature;
Theorem~\ref{thm:extension2}, which applies to universal
$C^*$~algebras such as the non-commutative torus or the Cuntz
algebras, is believed to be novel. Uniqueness of the solution is
proved, and it is also shown that $j$ is a cocycle, \ie it satisfies
the evolution equation
\begin{equation}\label{eqn:cocycle}
j_{s + t} = ( j_s \ctimes \iota_{\bop{\fock_{[ s ,\infty )}}{}} ) %
\comp \sigma_s \comp j_t \qquad \text{for all } s, t \ge 0,
\end{equation}
where $( \sigma_t )_{t \ge 0}$ is the shift semigroup on the algebra
of all bounded operators on~$\fock$. At this point we see another
novel feature of our work in contrast to previous results, all of
which start with a particular quantum dynamical semigroup $T$. In
these other papers the generator~$\tau$ of~$T$ is then augmented to
produce $\phi$, and the QSDE solved to give a dilation of $T$. For
example, in~\cite{FaS93} it is assumed that $T$ is an analytic
semigroup and that the composition of $\tau$ with the other components
of $\phi$ is well behaved in a certain sense; in~\cite{SiG07} it is
assumed that $T$ is covariant with respect to some group action on
$\alg$. For us, the starting point is the map $\phi$, which yields the
cocycle $j$, and hence, by compression, a quantum dynamical semigroup
$T$ generated by the closure of $\tau$, which has core $\alg_0$; this
semigroup, \textit{a fortiori}, is dilated by $j$. Thus we do not have
to check that $\tau$ is a semigroup generator with good properties at
the outset, thereby rendering our method easier to apply.

Our first application of Theorem~\ref{thm:extension1}, in
Section~\ref{sec:walks}, is to construct the Markov semigroups which
correspond to certain random walks on discrete groups.
Theorem~\ref{thm:extension1} is also employed in
Section~\ref{sec:sqep} to produce a dilation of the symmetric quantum
exclusion semigroup. This object, a model for systems of interacting
quantum particles, was introduced by Rebolledo \cite{Reb05} as a
non-commutative generalisation of the classical exclusion process
\cite{Lig99} and has generated much interest: see~\cite{PMQ09}
and~\cite{GQPM11}. The multiplicity space $\mul$ is required to be
infinite dimensional for this process, as in previous work on
processes arising from quantum interacting particle systems, \eg
\cite{GSS05}.

In Section~\ref{sec:uni} we use Theorem~\ref{thm:extension2} to obtain
flows on some universal $C^*$~algebras, namely the non-commutative
torus and the universal rotation algebra \cite{AnP89}; the former is a
particularly important example in non-commutative geometry. Quantum
flows on these algebras have previously been considered by Goswami,
Sahu and Sinha \cite{GSS05} and by Hudson and Robinson \cite{HuR88},
respectively.

\subsection{Conventions and notation}

The quantity $:=$ is to be read as `is defined to be' or similarly.
The quantity $\tfn{P}$ equals $1$ if the proposition $P$ is true and
$0$ if $P$ is false, where $1$ and $0$ are the appropriate
multiplicative and additive identities. The set of natural numbers is
denoted by $\N := \{ 1, 2, 3, \ldots \}$; the set of non-negative
integers is denoted by $\Z_+ := \{ 0, 1, 2, \ldots \}$. The linear
span of the set $S$ in the vector space $V$ is denoted by $\lin S$;
all vector spaces have complex scalar field and inner products are
linear on the right. The \emph{algebraic} tensor product is denoted by
$\otimes$; the Hilbert-space tensor product is denoted by $\ctimes$,
as is the ultraweak tensor product. The domain of the linear
operator~$T$ is denoted by $\dom T$. The identity transformation on
the vector space~$V$ is denoted by~$\id_V$. If $P$ is an orthogonal
projection on the inner-product space $V$ then the complement
$P^\perp := \id_V - P$, the projection onto the orthogonal complement
of the range of $P$. The Banach space of bounded operators from the
Banach space $X_1$ to the Banach space $X_2$ is denoted by
$\bop{X_1}{X_2}$, or by~$\bop{X_1}{}$ if $X_1$ and $X_2$ are equal.
The identity automorphism on the algebra $\alg$ is denoted
by~$\iota_\alg$. If $a$ and $b$ are elements in an algebra $\alg$ then
$[ a, b ] := a b - b a$ and $\{ a, b \} := a b + b a$ denote their
commutator and anti-commutator, respectively. If $\alg_0$ is a
$*$-algebra, $\hilb_1$ and $\hilb_2$ are Hilbert spaces and
$\alpha: \alg_0 \to \bop{\hilb_1}{\hilb_2}$ is a linear map then the
\emph{adjoint} map $\alpha^\dag: \alg_0 \to \bop{\hilb_2}{\hilb_1}$ is
such that $\alpha^\dag( a ) := \alpha( a^* )^*$ for all
$a \in \alg_0$.

\section{A higher-order product formula}\label{sec:hopf}

\begin{notation}
The Dirac bra-ket notation will be useful: for any Hilbert
space $\hilb$ and vectors $\xi$, $\chi \in \hilb$, let
\begin{alignat*}{3}
\ket{\hilb} & := \bop{\C}{\hilb}, & \qquad & %
\ket{\xi} : \C \to \hilb; \ \lambda \mapsto \lambda \xi %
& \qquad \text{(\emph{ket})\hphantom{.}} \\[2ex]
\text{and} \qquad %
\bra{\hilb} & := \bop{\hilb}{\C}, & & %
\bra{\chi} : \hilb \to \C; \ %
\eta \mapsto \langle \chi, \eta \rangle & \qquad \text{(\emph{bra})}.
\end{alignat*}
\end{notation}
In particular, we have the linear map
$\dyad{\xi}{\chi} \in \bop{\hilb}{}$ such that
$\dyad{\xi}{\chi} \eta = \langle \chi, \eta \rangle \xi$ for all
$\eta \in \hilb$.

Let $\alg \subseteq \bop{\ini}{}$ be a unital $C^*$~algebra with
identity $\id = \id_\ini$, whose elements act as bounded operators on
the \emph{initial space} $\ini$, a Hilbert space. Let
$\alg_0 \subseteq \alg$ be a norm-dense $*$-subalgebra of~$\alg$ which
contains~$\id$.

Let the \emph{extended multiplicity space} $\mmul := \C \oplus \mul$,
where the \emph{multiplicity space} $\mul$ is a Hilbert space,
and distinguish the unit vector $\vac := ( 1, 0 )$. For
brevity, let $\bopp := \bop{\mmul}{}$.

Let $\Delta := \id \otimes P_\mul \in \alg_0 \otimes \bopp$, where
$P_\mul := \dyad{\vac}{\vac}^\perp \in \bopp$ is the orthogonal
projection onto $\mul \subset \mmul$.

\begin{lemma}\label{lem:gen}
The map $\phi : \alg_0 \to \alg_0 \otimes \bopp$ is
$*$-linear, such that $\phi( \id ) = 0$ and such that
\begin{equation}\label{eqn:gen}
\phi( x y ) = %
\phi( x ) ( y \otimes \id_\mul ) + %
( x \otimes \id_\mul ) \phi( y ) + %
\phi( x ) \Delta \phi( y )
\qquad \text{for all } x, y \in \alg_0
\end{equation}
if and only if
\begin{equation}\label{eqn:gencpts}
\phi( x ) = \begin{bmatrix} \tau( x ) & \delta^\dagger( x ) \\[1ex]
\delta( x ) & \pi( x ) - x \otimes \id_\mul \end{bmatrix}
\qquad \text{for all } x \in \alg_0,
\end{equation}
where $\pi : \alg_0 \to \alg_0 \otimes \bop{\mul}{}$ is a
unital $*$-homomorphism, 
$\delta : \alg_0 \to \alg_0 \otimes \ket{\mul}$ is a
$\pi$-derivation, \ie a linear map such that
\[
\delta( x y ) = \delta( x ) y + \pi( x ) \delta( y ) %
\qquad \text{for all } x, y \in \alg_0,
\]
and $\tau : \alg_0 \to \alg_0$ is a $*$-linear map such that
\begin{equation}\label{eqn:tau}
\tau( x y ) - \tau( x ) y - x \tau( y ) = %
\delta^\dagger( x ) \delta( y ) %
\qquad \text{for all } x, y \in \alg_0.
\end{equation}
\end{lemma}
\begin{proof}
This is a straightforward exercise in elementary algebra.
\end{proof}

\begin{definition}
A $*$-linear map $\phi : \alg_0 \to \alg_0 \otimes \bopp$ such that
$\phi( \id ) = 0$ and such that~\eqref{eqn:gen} holds is a
\emph{flow generator}.
\end{definition}

\begin{remark}\label{rem:BEcdc}
Condition~\eqref{eqn:tau} may be expressed in terms of the
Bakry--\'{E}mery \textit{carr\'{e} du champ} operator
\[
\Gamma : \alg_0 \times \alg_0 \to \alg_0; \ ( x, y ) \mapsto %
\hlf \bigl( \tau( x y ) - \tau( x ) y - x \tau ( y ) \bigr);
\]
for \eqref{eqn:tau} to be satisfied, it is necessary and sufficient
that $2 \Gamma( x, y ) = \delta^\dagger( x ) \delta( y )$ for all
$x$, $y \in \alg_0$.

The $\pi$-derivation $\delta$ becomes a bimodule derivation if $\alg_0
\otimes \ket{\mul}$ is made into an $\alg_0$-$\alg_0$ bimodule by
setting $x \cdot z \cdot y := \pi( x ) z y$ for all $x$, $y \in
\alg_0$ and $z \in \alg_0 \otimes \ket{\mul}$.
\end{remark}

\begin{lemma}\label{lem:bddphi}
Let $\alg_0 = \alg$, let
$\pi : \alg \to \alg \otimes \bop{\mul}{}$ be a unital
$*$-homomorphism, let $z \in \alg \otimes \ket{\mul}$ and let
$h \in \alg$ be self adjoint. Define
\[
\delta : \alg \to \alg \otimes \ket{\mul}; \ %
x \mapsto z x - \pi( x ) z
\]
and
\[
\tau : \alg \to \alg; \ %
x \mapsto \I [ h, x ] - \hlf \{ z ^* z, x \} + z^* \pi( x ) z.
\]
Then the map $\phi: \alg \to \alg \otimes \bopp$ defined in terms of
$\pi$, $\delta$ and $\tau$ through~\eqref{eqn:gencpts} is a flow
generator.
\end{lemma}
\begin{proof}
This is another straightforward exercise.
\end{proof}

\begin{remark}
Modulo important considerations regarding tensor products and the
ranges of $\delta$ and $\tau$, the above form for $\phi$ is,
essentially, the only one possible \cite[Lemma~6.4]{LiW00b}. The
quantum exclusion process in Section~\ref{sec:sqep} has a generator of
the same form but with unbounded $z$ and $h$.
\end{remark}

\begin{definition}\label{dfn:qrw}
Given a flow generator $\phi : \alg_0 \to \alg_0 \otimes \bopp$, the
\emph{quantum random walk} $\bigl( \phi_n \bigr)_{n \in \Z_+}$ is a
family of $*$-linear maps
\[
\phi_n : \alg_0 \to \alg_0 \otimes \bopp^{\otimes n}
\]
defined by setting
\[
\phi_0 := \iota_{\alg_0} \qquad \text{and} \qquad %
\phi_{n + 1} := %
\bigl( \phi_n \otimes \iota_{\bopp} \bigr) \comp \phi %
\qquad \text{for all } n \in \Z_+. 
\]
The following identity is useful: if $\xi_1$, $\chi_1$, \ldots,
$\xi_n$, $\chi_n \in \mmul$ and $x \in \alg_0$ then
\begin{equation}\label{eqn:ncpts}
\bigl( \id_\ini \otimes \bra{\xi_1} \otimes \cdots \otimes \bra{\xi_n}
\bigr)
\phi_n( x )
\bigl( \id_\ini \otimes \ket{\chi_1} \otimes \cdots \otimes %
\ket{\chi_n} \bigr) = %
\phi^{\xi_1}_{\chi_1} \comp \cdots \comp \phi^{\xi_n}_{\chi_n}( x ),
\end{equation}
where
\[
\phi^\xi_\chi : \alg_0 \to \alg_0; \ x \mapsto %
( 1_\ini \otimes \bra{\xi} ) \phi (x) (1_\ini \otimes \ket{\chi} )
\]
is a linear map for each choice of $\xi$, $\chi \in \mmul$.
\end{definition}

\begin{remark}\label{rem:ordering}
The paper \cite{LiW03}, results from which will be employed below,
uses a different convention to that adopted in
Definition~\ref{dfn:qrw}: the components of the product
$\bopp^{\otimes n}$ appear in the reverse order to how they do above.
\end{remark}

\begin{notation}
Let $\alpha \subseteq \{ 1, \ldots, n \}$, with elements arranged in
increasing order, and denote its cardinality by~$| \alpha |$. The
unital $*$-homomorphism
\[
\alg_0 \otimes \bopp^{\otimes | \alpha |} \to %
\alg_0 \otimes \bopp^{\otimes n}; \ T \mapsto T( n, \alpha )
\]
is defined by linear extension of the map
\[
A \otimes B_1 \otimes \cdots \otimes B_{| \alpha |} \mapsto %
A \otimes C_1 \otimes \cdots \otimes C_n,
\]
where
\[
C_i := \left\{ \begin{array}{ll}
 B_j & \text{if $i$ is the $j$th element of $\alpha$}, \\[1ex]
 \id_\mmul & \text{if $i$ is not an element of $\alpha$}.
\end{array}\right.
\]
For example, if $\alpha = \{ 1, 3, 4 \}$ and $n = 5$ then
\[
( A \otimes B_1 \otimes B_2 \otimes B_3 )( 5, \alpha ) = %
A \otimes %
B_1 \otimes \id_\mmul \otimes B_2 \otimes B_3 \otimes \id_\mmul.
\]

Given a flow generator $\phi : \alg_0 \to \alg_0 \otimes \bopp$, for
all $n \in \Z_+$ and $\alpha \subseteq \{ 1, \ldots, n \}$,
let
\[
\phi_{| \alpha |}( x; n, \alpha ) := %
\bigl( \phi_{| \alpha |}( x ) \bigr)( n, \alpha ) %
\qquad \text{for all } x \in \alg_0
\]
and let
\[
\Delta( n, \alpha ) := %
( \id_\ini \otimes P_\mul^{\otimes | \alpha|} )( n, \alpha ),
\]
so that, in the latter, $P_\mul$ acts on the components of
$\mmul^{\otimes n}$ which have indices in $\alpha$ and $\id_\mmul$
acts on the others.
\end{notation}

\begin{theorem}
Let $\bigl( \phi_n \bigr)_{n \in \Z_+}$ be the quantum random walk
given by the flow generator $\phi$. For all $n \in \Z_+$ and
$x$,~$y \in \alg_0$,
\begin{equation}\label{eqn:hoprod}
\phi_n( x y ) = %
\sum_{\alpha \cup \beta = \{ 1, \ldots, n \} } %
\phi_{| \alpha |}( x ; n, \alpha ) \Delta( n, \alpha \cap \beta ) %
\phi_{| \beta |}( y ; n, \beta),
\end{equation}
where the summation is taken over all sets $\alpha$ and $\beta$ whose
union is $\{ 1, \ldots n \}$.
\end{theorem}
\begin{proof}
This may be established inductively: see
\cite[Proof of Theorem~4.1]{LiW03}.
\end{proof}

\begin{definition}
The set $S \subseteq \alg_0$ is \emph{$*$-generating for $\alg_0$} if
$\alg_0$ is the smallest unital $*$-algebra which contains~$S$.
\end{definition}

\begin{corollary}\label{cor:bound}
For a flow generator $\phi : \alg_0 \to \alg_0 \otimes \bopp$, let
\begin{equation}\label{eqn:growth}
\alg_\phi := \{ x \in \alg_0 : %
\text{there exist $C_x$,~$M_x > 0$ such that
$\| \phi_n( x ) \| \le C_x M_x^n$ for all } n \in \Z_+ \}.
\end{equation}
Then $\alg_\phi$ is a unital $*$-subalgebra of $\alg_0$, which
is equal to $\alg_0$ if $\alg_\phi$ contains a $*$-generating set
for~$\alg_0$.
\end{corollary}
\begin{proof}
It suffices to demonstrate that $\alg_\phi$ is closed under products. To
see this, let $x$,~$y \in \alg_\phi$ and suppose
$C_x$, $M_x$ and $C_y$, $M_y$ are as in~\eqref{eqn:growth}.
Then~\eqref{eqn:hoprod} implies that
\begin{align*}
\| \phi_n( x y ) \| & \le %
\sum_{\alpha \cup \beta = \{ 1, \ldots, n \} } %
\| \phi_{| \alpha |}( x ) \| \, %
\| \phi_{| \beta |}( y ) \| \\[1ex]
 & \le C_x C_y \sum_{k = 0}^n \binom{n}{k} M_x^k %
\sum_{l = 0}^k \binom{k}{l} M_y^{n - k + l} \qquad %
(k = | \alpha |,\ l = | \alpha \cap \beta |) \\[1ex]
 & = C_x C_y \sum_{k = 0}^n %
\binom{n}{k} M_x^k M_y^{n - k} ( 1 + M_y )^k \\[1ex]
 & = C_x C_y ( M_x + M_x M_y + M_y )^n
\end{align*}
for all $n \in \Z_+$, as required.
\end{proof}

\begin{lemma}
If the flow generator $\phi$ is as defined in
Lemma~\ref{lem:bddphi} then $\alg_\phi = \alg_0$.
\end{lemma}
\begin{proof}
This follows immediately, since $\phi$ is completely bounded and
$\| \phi_n \| \le \| \phi_n \|_\cb \le \| \phi \|_\cb^n$ for
all~$n \in \Z_+$.
\end{proof}

The following result shows that, given a flow generator $\phi$ and
vectors $\chi$, $\xi \in \mmul$, the elements of~$\alg_\phi$ are
entire vectors for $\phi^\xi_\chi$.

\begin{lemma}\label{lem:abe}
Let $\phi : \alg_0 \to \alg_0 \otimes \bopp$ be a flow generator. For
all $\xi$,~$\chi \in \mmul$ we have
$\phi_\chi^\xi( \alg_\phi ) \subseteq \alg_\phi$, and the series
\begin{equation}\label{eqn:series}
\exp( z \phi_\chi^\xi ) := %
\sum_{n = 0}^\infty \frac{z^n ( \phi_\chi^\xi )^n}{n!}
\end{equation}
is strongly absolutely convergent on $\alg_\phi$ for all $z \in \C$.
\end{lemma}
\begin{proof}
Suppose $\| \phi_n( x ) \| \le C_x M_x^n$ for all $n \in \Z_+$.
It follows from~\eqref{eqn:ncpts} that
\begin{equation}\label{eqn:qrwbk}
\bigl( \id_{\ini \ctimes \mmul^{\ctimes n}} \otimes \bra{\xi} \bigr) %
\phi_{n + 1}( x ) %
\bigl( \id_{\ini \ctimes \mmul^{\ctimes n}} \otimes \ket{\chi} %
\bigr) = %
\phi_n\bigl( \phi_\chi^\xi( x ) \bigr),
\end{equation}
so
\[
\bigl\| \phi_n \bigl( \phi_\chi^\xi( x ) \bigr) \bigr\| \le %
\| \xi \| C_x M_x^{n + 1} \| \chi \| = %
( \| \xi \| \, \| \chi \| C_x M_x ) M_x^n
\]
and $\phi_\chi^\xi( x ) \in \alg_\phi$. Moreover~\eqref{eqn:ncpts}
also gives that
\begin{equation}\label{eqn:qrwslicebound}
\bigl\| \bigl( \phi_{\chi_1}^{\xi_1} \comp \cdots \comp %
\phi_{\chi_n}^{\xi_n} \bigr)( x ) \bigr\| \le %
\| \xi_1 \| \cdots \| \xi_n \| \| \chi_1 \| \cdots %
\| \chi_n \| C_x M_x^n,
\end{equation}
hence the series~\eqref{eqn:series} converges as claimed.
\end{proof}

\section{Quantum flows}\label{sec:qf}

\begin{notation}
Let $\fock$ denote Boson Fock space over $\elltwo$, the Hilbert space
of $\mul$-valued, square-integrable functions on the half line, and
let
\[
\evecs := \lin\{ \evec{f} : f \in \elltwo \}
\]
denote the linear span of the total set of exponential vectors
in $\fock$. As is customary, elementary tensors in
$\ini \otimes \fock$ are written without a tensor-product sign: in
other words, $u \evec{f} := u \otimes \evec{f}$ for all $u \in \ini$
and $f \in \elltwo$, \etc.

If $f \in \elltwo$ and $t \ge 0$ then $\wh{f}( t ) := \wh{f( t )}$,
where $\wh{\xi} := \vac + \xi \in \mmul$ for all $\xi \in \mul$.

Given $f \in \elltwo$ and an interval $I \subseteq \R_+$, let
$f_I \in \elltwo$ be defined to equal $f$ on~$I$ and~$0$ elsewhere,
with $f_{t)} := f_{[ 0, t )}$ and $f_{[t} := f_{[ t, \infty )}$ for all
$t \ge 0$.
\end{notation}

\begin{definition}
A family of linear operators $( T_t )_{t \ge 0}$ in
$\ini \ctimes \fock$ with domains including $\ini \otimes \evecs$ is
\emph{adapted} if
\[
\langle u \evec{f}, T_t v \evec{g} \rangle = %
\langle u \evec{f_{t)}}, T_t v \evec{g_{t)}} \rangle %
\langle \evec{f_{[t}}, \evec{g_{[t}} \rangle
\]
for all $u$,~$v \in \ini$, $f$,~$g \in \elltwo$ and~$t \ge 0$.
\end{definition}

\begin{theorem}\label{thm:qwi}
For all $n \in \N$ and
$T \in \bop{\ini \ctimes \mmul^{\ctimes n}}{}$ there exists a
family $\Lambda^n( T ) = \bigl( \Lambda^n_t( T ) \bigr)_{t \ge 0}$ of
linear operators in $\ini \ctimes \fock$, with domains including
$\ini \otimes \evecs$, that is adapted and
such that
\begin{equation}\label{eqn:qwiip}
\langle u \evec{f}, \Lambda^n_t( T ) v \evec{g} \rangle = %
\int_{\simp_n( t )} %
\langle u \otimes \wh{f}^{\otimes n}( \bt ), %
 T v\otimes \wh{g}^{\otimes n}( \bt ) \rangle \std \bt \,
\langle \evec{f}, \evec{g} \rangle
\end{equation}
for all $u$,~$v \in \ini$, $f$,~$g \in \elltwo$ and $t \ge 0$. Here
the simplex
\[
\simp_n( t ) := \{ \bt := ( t_1, \ldots, t_n ) \in [ 0 , t ]^n : %
t_1 < \cdots < t_n \}
\]
and
\[
\wh{f}^{\otimes n}( \bt ) := %
\wh{f}( t_1 ) \otimes \cdots \otimes \wh{f}( t_n ),
\qquad \text{\etc}.
\]
We extend this definition to include $n = 0$ by setting
$\Lambda^0_t( T ) := T \otimes \id_\fock$ for all $t \ge 0$.

If $f \in \elltwo$ then
\begin{equation}\label{eqn:qwi}
\| \Lambda^n_t( T ) u \evec{f} \| \le %
\frac{K_{f, t}^n}{\sqrt{n!}} \, \| T \| \, \| u \evec{f} \| %
\qquad \text{for all }t \ge 0 \text{ and } u \in \ini,
\end{equation}
where
$K_{f, t} := \sqrt{(2 + 4 \| f \|^2 ) ( t + \| f \|^2 )}$,
and the map
\[
\R_+ \to \bop{\ini}{\ini \ctimes \fock}; \ %
t \mapsto \Lambda^n_t( T ) %
\bigl( \id_\ini \otimes \ket{\evec{f}} \bigr)
\]
is norm continuous.
\end{theorem}
\begin{proof}
This is an extension of Proposition~3.18 of~\cite{Lin05}, from which
we borrow the notation; as for Remark~\ref{rem:ordering}, the
ordering of the components in tensor products is different but this is
no more than a convention. For each $f \in \elltwo$ define
$C_f \ge 0$ so that
\[
C_f^2 = \bigl( \| f \| + \sqrt{1 + \| f \|^2} \bigr)^2 \le %
2 + 4 \| f \|^2,
\]
and note that, by inequality~(3.21) of \cite{Lin05},
\begin{align*}
\| \Lambda^n_t ( T ) u \evec{f} \|^2 & \le %
\bigl( C_{f_{t)}} \bigr)^{2n} \int_{\simp_n(t)} %
\| T u \otimes \wh{f}^{\otimes n}( \bt ) \|^2 \std \bt \, %
\| \evec{f} \|^2 \\
& \le %
\frac{K_{f, t}^{2 n}}{n!} \, \| T \|^2 \| u \evec{f} \|^2.
\end{align*}
To show continuity, let $\wt{T}$ denote $T$ considered as an
operator on
$( \ini \ctimes \mmul ) \ctimes \mmul^{\ctimes ( n - 1 )}$, where the
right-most copy of $\mmul$ in the $n$-fold tensor product has moved
next to the initial space $\ini$. Then
\[
\Lambda^n_t( T ) - \Lambda^n_s( T ) = %
\Lambda_t \bigl( %
1_{( s, t ]}( \cdot ) \Lambda^{n - 1}_\cdot( \wt{T} )
\bigr),
\]
and so, using Theorem~3.13 of~\cite{Lin05},
\begin{align*}
\| \bigl( \Lambda^n_t( T ) - \Lambda^n_s( T ) \bigr) u \evec{f} \|^2
& \le 2 ( t + C_f^2 ) %
\int^t_s \| \Lambda^{n - 1}_r( \wt{T} ) %
\bigl( u \otimes \wh{f}( r ) \bigr) \evec{f} \|^2 \std r \\
& \le 2 ( t + C_f^2 ) \Bigl( \int^t_s \| \wh{f} (r) \|^2 \std r \Bigr)
\frac{K_{f, t}^{2 n - 2}}{( n - 1 )!} \, \| T \|^2 %
\| u \evec{f} \|^2.
\qedhere
\end{align*}
\end{proof}

The family $\Lambda^n( T )$ is the
\emph{$n$-fold quantum Wiener integral} of~$T$.

\begin{remark}
It may be shown \cite[Proof of Theorem~2.2]{LiW03} that
\[
\dom \Lambda^l_t( S )^* \supseteq %
\Lambda^m_t( T )( \ini \otimes \evecs )
\]
for all $l$,~$m \in \Z_+$, 
$S \in \bop{\ini \ctimes \mmul^{\ctimes l}}{}$,
$T \in \bop{\ini \ctimes \mmul^{\ctimes m}}{}$ and $t \ge 0$.
\end{remark}

\begin{theorem}\label{thm:intrep}
Let $\phi : \alg_0 \to \alg_0 \otimes \bopp$ be a flow generator. If
$x \in \alg_\phi$ then the series
\begin{equation}\label{eqn:jdef}
j_t( x ) := %
\sum_{n = 0}^\infty \Lambda^n_t\bigl( \phi_n( x ) \bigr)
\end{equation}
is strongly absolutely convergent on $\ini \otimes \evecs$ for all
$t \ge 0$, uniformly so on compact subsets of $\R_+$. The map
\[
\R_+ \to \bop{\ini}{\ini \ctimes \fock}; \ %
t \mapsto j_t( x ) %
\bigl( \id_\ini \otimes \ket{\evec{f}} \bigr)
\]
is norm continuous for all $f \in \elltwo$, the family
$\bigl( j_t( x ) \bigr)_{t \ge 0}$ is adapted and
\begin{equation}\label{eqn:wqsde}
\langle u \evec{f}, j_t( x ) v\evec{g} \rangle = %
\langle u \evec{f}, ( x v ) \evec{g} \rangle + %
\int_0^t \bigl\langle u \evec{f}, %
j_s \bigl( \phi_{\wh{g}( s )}^{\wh{f}( s )}( x ) \bigr) %
v \evec{g} \bigr\rangle \std s
\end{equation}
for all $u$,~$v \in \ini$, $f$,~$g \in \elltwo$, $x \in \alg_\phi$ and
$t \ge 0$. Furthermore,
\begin{equation}\label{eqn:matsp}
\bigl( \id_\ini \otimes \bra{\evec{f}} \bigr) j_t( x ) %
\bigl( \id_\ini \otimes \ket{\evec{g}} \bigr) \in \alg
\end{equation}
for all~$x \in \alg_\phi$, $f$,~$g \in \elltwo$ and $t \ge 0$.
\end{theorem}
\begin{proof}
The first two claims are a consequence of the
estimate~\eqref{eqn:qwi}, the definition of~$\alg_\phi$ and the
continuity result from Theorem~\ref{thm:qwi}; adaptedness is inherited
from the adaptedness of the quantum Wiener integrals.
Lemma~\ref{lem:abe} implies that the integrand on the right-hand side
of~\eqref{eqn:wqsde} is well defined and,
by~\eqref{eqn:qrwbk},
\begin{multline*}
\bigl\langle u\evec{f}, %
\Lambda^n_s\bigl( \phi_n \bigl( %
\phi_{\wh{g}( s )}^{\wh{f}( s )}( x ) \bigr) \bigr) %
v \evec{g} \bigr\rangle \\[1ex]
\begin{aligned}
& = \int_{\simp_n( s )} %
\langle u \otimes \wh{f}^{\otimes n}( \bt ), %
\phi_n\bigl( \phi_{\wh{g}( s )}^{\wh{f}( s )}( x ) \bigr) %
v \otimes \wh{g}^{\otimes n}( \bt ) \rangle %
\std \bt \langle \evec{f}, \evec{g} \rangle \\[1ex]
& = \int_{\simp_n( s )} \langle u \otimes %
\wh{f}^{\otimes n}( \bt ) \otimes \wh{f}( s ), %
\phi_{n + 1}( x ) v \otimes \wh{g}^{\otimes n}( \bt ) \otimes %
\wh{g}( s ) \rangle \std \bt \langle \evec{f}, \evec{g} \rangle;
\end{aligned}
\end{multline*}
integrating with respect to $s$ then taking the sum of these terms
gives~\eqref{eqn:wqsde}. For the final claim, note that for any
$f$, $g \in \elltwo$, the $\alg_0$-valued map
\[
\simp_n( t ) \ni \bt \mapsto %
\phi^{\wh{f}( t_1 )}_{\wh{g}( t_1 )} \comp \cdots \comp %
\phi^{\wh{f}( t_n )}_{\wh{g}( t_n )}( x ) = %
\bigl( \id_\ini \otimes \bra{ \wh{f}^{\otimes n}( \bt ) } \bigr) %
\phi_n( x )
\bigl( \id_\ini \otimes \ket{ \wh{g}^{\otimes n}( \bt ) } \bigr)
\]
is Bochner integrable, hence
\begin{equation}\label{eqn:cdqwi}
\bigl( \id_\ini \otimes \bra{\evec{f}} \bigr) %
\Lambda_t^n\bigl( \phi_n( x ) \bigr) %
\bigl( \id_\ini \otimes \ket{\evec{g}} \bigr) = %
e^{\langle f, g \rangle} \int_{\simp_n( t )} %
\bigl( \phi_{\wh{g}( t_1 )}^{\wh{f}( t_1 )} %
\comp \cdots \comp %
\phi_{\wh{g}( t_n )}^{\wh{f}( t_n )} \bigr)( x ) \std \bt %
\in \alg.
\end{equation}
By~\eqref{eqn:qrwslicebound}, we may sum~\eqref{eqn:cdqwi} over all
$n \in \Z_+$, with the resulting series being norm convergent, and so
the final claim follows.
\end{proof}

\begin{remark}\label{rmk:jid}
For all $t \ge 0$, let $j_t$ be as in Theorem~\ref{thm:intrep}.
Since $\alg_\phi$ is a subspace of $\alg_0$ containing~$1$, and
each $\phi_n$ is linear with $\phi_n( 1 ) = 0$, it follows
from~\eqref{eqn:jdef} and Theorem~\ref{thm:qwi} that each $j_t$ is
linear and unital, as a map into the space of operators with domain
$\ini \otimes \evecs$. Moreover, the maps $j_t$ are weakly
$*$-homomorphic in the following sense.
\end{remark}

\begin{lemma}\label{lem:mult}
Let $\phi : \alg_0 \to \alg_0 \otimes \bopp$ be a flow generator and
let $j_t$ be as in Theorem~\ref{thm:intrep} for
all~$t \ge 0$. If $x$,~$y \in \alg_\phi$ then
$x^* y \in \alg_\phi$, with
\begin{equation}\label{eqn:wmult}
\langle j_t( x ) u \evec{f}, j_t( y ) v \evec{g} \rangle = %
\langle u \evec{f}, j_t( x^* y ) v \evec{g} \rangle
\end{equation}
for all $u$,~$v \in \ini$ and $f$,~$g \in \elltwo$. In particular, if
$x \in \alg_\phi$ then $j_t( x )^* \supseteq j_t( x^* )$.
\end{lemma}
\begin{proof}
As $\alg_\phi$ is a $*$-algebra, so $x^* y \in \alg_\phi$. Let
$N \in \Z_+$ and note that, by \cite[Theorem~2.2]{LiW03},
\begin{equation}\label{eqn:finiteIto}
\sum_{l, m = 0}^N \Lambda^l_t\bigl( \phi_l( x ) \bigr)^* %
\Lambda^m_t\bigl( \phi_m( y ) \bigr) = %
\sum_{n = 0}^{2 N} \Lambda^n_t\bigl( \phi_{n, N]}( x^* y ) \bigr) %
\qquad \text{on } \ini \otimes \evecs,
\end{equation}
where
\[
\phi_{n, N]}( x^* y ) := %
\sum_{\substack{\alpha \cup \beta = \{ 1, \ldots, n \}\\[0.5ex]
| \alpha |, \, | \beta | \le N }} %
\phi_{| \alpha |}( x^* ; n, \alpha ) \Delta( n, \alpha \cap \beta ) %
\phi_{| \beta |}( y ; n, \beta ).
\]
Working as in the proof of Corollary~\ref{cor:bound} yields
the inequality
\[
\| \phi_{n, N]}( x^* y ) \| \le %
C_{x^*} C_y ( M_{x^*} + M_{x^*} M_y + M_y )^n,
\]
and so, by \eqref{eqn:qwi},
\begin{equation}\label{eqn:cutoffest}
| \langle u \evec{f}, \Lambda^n_t\bigl( \phi_{n, N]}( x^* y ) \bigr) %
v \evec{g} \rangle | \le %
\frac{K_{g, t}^n ( M_{x^*} + M_{x^*} M_y + M_y )^n}{\sqrt{n!}} \, %
C_{x^*} C_y \, \| u \evec{f} \| \, \| v \evec{g} \|.
\end{equation}
As $\phi_{n,N]} = \phi_n$ if $n \in \{0, 1, \ldots, N\}$, it follows
that
\begin{align*}
\langle j_t( x ) u \evec{f}, j_t( y ) v \evec{g} \rangle & = %
\lim_{N \to \infty} \sum_{l, m = 0}^N \langle u \evec{f}, %
\Lambda^l_t\bigl( \phi_l( x ) \bigr)^* %
\Lambda^m_t\bigl( \phi_m( y ) \bigr) v \evec{g} \rangle
\\[1ex]
 & = \lim_{N \to \infty} \sum_{n = 0}^N \langle u \evec{f}, %
\Lambda^n_t\bigl( \phi_n( x^* y ) \bigr) v \evec{g} \rangle
\\[1.5ex]
 & \hspace{5em} + \lim_{N \to \infty} %
\smash[t]{\sum_{n = N + 1}^{2 N}} \langle u \evec{f}, %
\Lambda^n_t\bigl( \phi_{n, N]}( x^* y ) \bigr) v \evec{g} \rangle
\\[1ex]
 & = \langle u \evec{f}, j_t( x^* y ) v \evec{g} \rangle,
\label{eqn:ll}
\end{align*}
since the final limit is zero by~\eqref{eqn:cutoffest}.
\end{proof}

\begin{lemma}\label{lem:uniquesoln}
If $\alg_\phi$ is dense in $\alg$ then there is at most one family of
$*$-homomorphisms $( \jt_t )_{t \ge 0}$ from $\alg$ to
$\bop{\ini \ctimes \fock}{}$ that satisfies~\eqref{eqn:wqsde}.
\end{lemma}
\begin{proof}
Suppose that $j^{( 1 )}$ and $j^{( 2 )}$ are two families of
$*$-homomorphisms from $\alg$ to $\bop{\ini \otimes \fock}{}$
that satisfy~\eqref{eqn:wqsde}. Set $k_t := j^{( 1 )}_t - j^{( 2 )}_t$
and note we have that
\[
\langle u \evec{f}, k_t( x ) v \evec{g} \rangle = %
\int_0^t \langle u \evec{f}, %
k_s \bigl( \phi^{\wh{f}( s )}_{\wh{g}( s )}( x ) \bigr) v \evec{g} %
\rangle \std s
\]
for all $u$, $v \in \ini$, $f$, $g \in \elltwo$ and $x \in \alg_\phi$.
Iterating the above, and using the fact that $\| k_t \| \le 2$ for all
$t \ge 0$, we obtain the inequality
\[
| \langle u \evec{f}, k_t( x ) v \evec{g} \rangle | \le 2 %
\int_{\simp_n( t )} \| \phi^{\wh{f}( t_1 )}_{\wh{g}( t_1 )} \comp %
\cdots \comp \phi^{\wh{f}( t_n )}_{\wh{g}( t_n )}( x ) \| \std \bt %
\, \| u \evec{f}\| \, \| v \evec{g} \|.
\]
However~\eqref{eqn:qrwslicebound} now gives that
\[
| \langle u \evec{f}, k_t (x) v \evec{g} \rangle | \le 2 C_x %
\frac{\bigl( M_x \| \wh{f}_{t)} \| \|\wh{g}_{t)} \| \bigr)^n}{n!} %
\, \| u \evec{f} \| \, \| v \evec{g} \|
\]
and the result follows by letting $n \to \infty$.
\end{proof}

\begin{theorem}\label{thm:extension1}
Let $\phi : \alg_0 \to \alg_0 \otimes \bopp$ be a flow generator and
suppose $\alg_0$ contains its square roots: for all non-negative
$x \in \alg_0$, the square root~$x^{1 / 2}$ lies in $\alg_0$. If
$\alg_\phi = \alg_0$ then, for all~$t \ge 0$, there exists a
unital $*$-homomorphism
\[
\jt_t : \alg \to \bop{\ini \ctimes \fock}{}
\]
such that $\jt_t( x ) = j_t( x )$ on $\ini \otimes \evecs$ for all
$x \in \alg_0$, where $j_t( x )$ is as defined in
Theorem~\ref{thm:intrep}.
\end{theorem}
\begin{proof}
Let $x \in \alg_0$ and suppose first that $x \ge 0$. If
$y := ( \| x \| \id - x )^{1 / 2}$, which lies in $\alg_0$ by
assumption, then Lemma~\ref{lem:mult} and Remark~\ref{rmk:jid} imply
that
\[
0 \le \| j_t( y ) \theta \|^2 = %
\langle \theta, j_t( y^2 ) \theta \rangle = %
\| x \| \, \| \theta \|^2 - \langle \theta, j_t( x ) \theta \rangle %
\qquad \text{for all } \theta \in \ini \otimes \evecs.
\]
If $x$ is now an arbitrary element of $\alg_0$, it follows that
\[
\| j_t( x ) \theta \|^2 = %
\langle \theta, j_t( x^* x ) \theta \rangle \le %
\| x^* x \| \, \| \theta \|^2 = \| x \|^2 \| \theta \|^2.
\]
Thus $j_t( x )$ extends to
$\jt_t( x ) \in \bop{\ini \ctimes \fock}{}$, which has norm at most
$\| x \|$, and the map
\[
\alg_0 \to \bop{\ini \ctimes \fock}{}; \ x \mapsto \jt_t( x )
\]
is a $*$-linear contraction, which itself extends to a $*$-linear
contraction
\[
\jt_t : \alg \to \bop{\ini \ctimes \fock}{}.
\]
Furthermore, if $x$,~$y \in \alg_0$ and
$\theta$,~$\zeta \in \ini \otimes \evecs$ then, by
Lemma~\ref{lem:mult},
\[
\langle \theta, \jt_t( x ) \jt_t( y ) \zeta \rangle = %
\langle \jt_t( x^* ) \theta, \jt( y ) \zeta \rangle = %
\langle j_t( x^* ) \theta, j_t( y ) \zeta \rangle = %
\langle \theta, j_t( x y ) \zeta \rangle = %
\langle \theta, \jt_t( x y ) \zeta \rangle,
\]
so $\jt_t$ is multiplicative on $\alg_0$. An approximation argument
now gives that $\jt_t$ is multiplicative on the whole of~$\alg$.
\end{proof}

\begin{remark}\label{rmk:af}
If $\alg$ is an AF algebra, \ie the norm closure of an increasing
sequence of finite-dimensional $*$-subalgebras, then its local algebra
$\alg_0$, the union of these subalgebras, contains its square roots,
since every finite-dimensional $C^*$~algebra is closed in $\alg$.
\end{remark}

\begin{definition}
The unital $C^*$~algebra $\alg$ has \emph{generators}
$\{ a_i : i \in I \}$ if $\alg$ is the smallest unital $C^*$~algebra
which contains $\{ a_i : i \in I \}$. These generators
\emph{satisfy the relations} $\{ p_k : k \in K \}$ if each $p_k$ is a
complex polynomial in the non-commuting indeterminate
$\langle X_i, X_i^* : i \in I \rangle$ and, for all $k \in K$, 
the algebra element $p_k( a_i , a_i^* : i \in I )$, obtained from
$p_k$ by replacing $X_i$ by $a_i$ and~$X_i^*$ by $a_i^*$ for all
$i \in I$, is equal to $0$.

Suppose $\alg$ has generators $\{ a_i : i \in I \}$ which satisfy the
relations $\{ p_k : k \in K \}$. Then $\alg$ is
\emph{generated by isometries} if
$\{ X_i^* X_i - 1 : i \in I \} \subseteq \{ p_k : k \in K \}$ and
is \emph{generated by unitaries}
if~$\{ X_i^* X_i - 1, \ X_i X_i^* - 1 : i \in I \} \subseteq %
\{ p_k : k \in K \}$.
The algebra $\alg$ is \emph{universal} if, given any unital
$C^*$~algebra $\blg$ containing a set of elements
$\{ b_i : i \in I \}$ which satisfies the relations
$\{ p_k : k \in K \}$, \ie~$p_k( b_i, b_i^* : i \in I ) = 0$ for all
$k \in K$, there exists a unique $*$-homomorphism
$\pi : \alg \to \blg$ such that~$\pi( a_i ) = b_i$ for all $i \in I$.
\end{definition}

\begin{theorem}\label{thm:extension2}
Let $\alg$ be the universal $C^*$~algebra generated by isometries
$\{ s_i : i \in I \}$ which satisfy the relations
$\{ p_k : k \in K \}$, and let $\alg_0$ be the $*$-algebra generated
by $\{ s_i : i \in I \}$. If~$\phi : \alg_0 \to \alg_0 \otimes \bopp$
is a flow generator such that $\alg_\phi = \alg_0$ then, for all
$t \ge 0$, there exists a unital $*$-homomorphism
\[
\jt_t : \alg \to \bop{\ini \ctimes \fock}{}
\]
such that $\jt_t( x ) = j_t( x )$ on $\ini \otimes \evecs$ for all
$x \in \alg_0$, where $j_t( x )$ is as defined in
Theorem~\ref{thm:intrep}.
\end{theorem}
\begin{proof}
Remark~\ref{rmk:jid} and Lemma~\ref{lem:mult} imply that $j_t( s_i )$
is isometric and that $j_t( s_i^* )$ is contractive for all $i \in I$.
Repeated application of~\eqref{eqn:wmult} then shows that
$j_t( x )$ is bounded for each $x \in \alg_0$, and that $j_t$ extends to a
unital $*$-homomorphism from $\alg_0$ to
$\bop{\ini \ctimes \fock}{}$. Furthermore, the
set~$\{ j_t( s_i ) : i \in I \}$ satisfies the relations
$\{ p_k : k \in K \}$ so, by the universal nature of $\alg$, there
exists a $*$-homomorphism $\pi$ from $\alg$ into
$\bop{\ini \ctimes \fock}{}$ such that $\pi( s_i ) = j_t( s_i )$ for
all $i \in I$ and $\jt_t := \pi$ is as required.
\end{proof}

\begin{corollary}\label{cor:stsoln}
The family
$\bigl( \jt_t: \alg \to \bop{\ini \ctimes \fock}{} \bigr)_{t \ge 0}$
constructed in Theorems~\ref{thm:extension1}
and~\ref{thm:extension2} is a strong solution of the
QSDE~\eqref{eqn:EHqsde}.
\end{corollary}
\begin{proof}
Fix $x \in \alg_\phi$ and let
\begin{equation}\label{eqn:integrand}
L_t := \Sigma\bigl( ( \jt_t \otimes \iota_\bopp )( \phi( x ) ) \bigr)
\end{equation}
for all $t \ge 0$, where
$\Sigma : \bop{\ini \ctimes \fock \ctimes \mmul}{} \to %
\bop{\ini \ctimes \mmul \ctimes \fock}{}$
is the isomorphism that swaps the last two components of simple
tensors. If $f \in \elltwo$ then
\[
\| L_t u \otimes \wh{f}( t ) \otimes \evec{f} \| \le %
\| \phi( x ) \| \, \| \wh{f}( t ) \| \, \| u \evec{f} \|,
\]
so if $t \mapsto L_t u \otimes \wh{f}( t ) \otimes \evec{f}$
is strongly measurable then $t \mapsto L_t$ is quantum
stochastically integrable \cite[p.232]{Lin05} and $\jt$
satisfies the QSDE in the strong sense, since we already have
from~\eqref{eqn:wqsde} that it is a weak solution.

Now, Theorem~\ref{thm:intrep} implies that for each
$x \in \alg_\phi = \alg_0$ and $\theta \in \ini \otimes \evecs$ the
map $t \mapsto \jt_t( x ) \theta$ is continuous, hence so is
\[
t \mapsto ( \jt_t \otimes \iota_\bopp )( y \otimes T )%
( \theta \otimes \xi ) = \jt_t( y ) \theta \otimes T \xi
\]
for all $y \in \alg_0$, $T \in \bop{\mmul}{}$ and
$\xi \in \mmul$. As $\| L_t \| = \| \phi( x ) \|$
for all $t \ge 0$, it follows that $t \mapsto L_t$ and
$t \mapsto L^*_t$ are strongly continuous on
$\ini \ctimes \mmul \ctimes \fock$. Hence
$t \mapsto L_t (u \otimes \wh{f}( t ) \otimes \evec{f})$ is
separably valued and weakly measurable, so Pettis's theorem
gives the result.
\end{proof}

\begin{remark}
Property~\eqref{eqn:matsp} implies that the homomorphism $\jt_t$ given
by Theorems~\ref{thm:extension1} and~\ref{thm:extension2} takes values
in the matrix space $\alg \matten \bop{\fock}{}$ \cite{Lin05}.
\end{remark}

\begin{notation}
For all $t \ge 0$, $f$,~$g \in \elltwo$ and $a \in \alg$, let
\[
\jt_t[ f, g ]( a ) := %
\bigl( \id_\ini \otimes \bra{\evec{f_{t)}}} \bigr) %
\jt_t( a ) \bigl( \id_\ini \otimes \ket{\evec{g_{t)}}} \bigr).
\]
\end{notation}

\begin{theorem}\label{thm:cocycle}
The family of $*$-homomorphisms $( \jt_t )_{t \ge 0}$ given by
Theorems~\ref{thm:extension1} and
\ref{thm:extension2} forms a Feller
cocycle~\emph{\cite[Section~2.4]{LiW01}} for the shift
semigroup on~$\bop{\fock}{}$: for all $s$,~$t \ge 0$,
$f$,~$g \in \elltwo$ and~$a \in \alg$,
\[
\begin{array}{rl}
\text{\rm (i)} & \jt_0[ 0, 0 ]( a ) = a, \\[1ex]
\text{\rm (ii)} & \jt_t[ f, g ]( a ) \in \alg, \\[1ex]
\text{\rm (iii)} & t \mapsto \jt_t[ f, g ]( a )
\text{ is norm continuous} \\[1ex]
\text{and} \quad \text{\rm (iv)} & %
\jt_{s + t}[ f, g ] = \jt_s[ f, g ] %
\comp \jt_t[ f( \cdot + s ), g( \cdot + s ) ].
\end{array}
\]
Consequently, setting
\[
T_t( a ) := \jt_t[ 0, 0 ]( a ) = %
\bigl( \id_\ini \otimes \bra{\evec{0}} \bigr) \jt_t( a ) %
\bigl( \id_\ini \otimes \ket{\evec{0}} \bigr) %
\qquad \text{for all } a \in \alg
\]
gives a strongly continuous semigroup $T = ( T_t )_{t \ge 0}$ of
completely positive contractions on~$\alg$ such that
$T_t( x ) = \exp( t \phi_\vac^\vac )( x )$ for all $x \in \alg_0$ and
$t \ge 0$. In particular, $T_t( \id ) = \id$ for all $t \ge 0$
and~$\alg_0$ is a core for the generator of $T$.
\end{theorem}
\begin{proof} Properties (i) and (ii) are immediate consequences
of~\eqref{eqn:wqsde} and~\eqref{eqn:matsp} respectively. For (iii),
note that if $x \in \alg_0$ and $f$,~$g \in \elltwo$ then
Theorem~\ref{thm:intrep} implies that
\[
t \mapsto \jt_t[ f, g ]( x ) = %
\bigl( \id_\ini \otimes \bra{\evec{f}} \bigr) j_t( x ) %
\bigl( \id_\ini \otimes \ket{\evec{g}} \bigr) %
\exp\Bigl( -\int_t^\infty \langle f( s ), g( s ) \rangle \std s \Bigr)
\]
is norm continuous; the general case follows by approximation.

In order to establish (iv), fix~$s \ge 0$ and continuous functions
$f$,~$g \in \elltwo$, and let
\[
J_t := \jt_s[ f, g ] \comp \jt_t[ f( \cdot + s ), g( \cdot + s ) ]
\qquad \text{for all } t \ge 0.
\]
We will show that $J_t = \jt_{s + t}[ f, g ]$.

First note that for any $x \in \alg_0$ and $t > 0$, the map
\[
F : [ 0, t ] \to \alg; \ r \mapsto \jt_r%
[ f( \cdot + s ), g( \cdot + s ) ] \bigl( %
\phi^{\wh{f}( r + s )}_{\wh{g}( r + s )}( x ) \bigr) %
\langle \evec{f_{[ s + r , s + t )}}, \evec{g_{[ s + r, s + t )}}
\rangle
\]
is continuous, hence Bochner integrable, and so
\[
x \langle \evec{f_{[ s, s + t )}}, \evec{g_{[ s, s + t )}} \rangle + %
\int^t_0 F( r ) \std r \in \alg.
\]
By the adaptedness of $\jt_t( x )$ and~\eqref{eqn:wqsde},
\begin{multline*}
\langle u, \Bigl( x \langle \evec{f_{[ s, s + t )}},
\evec{g_{[ s, s + t )}} \rangle + \int^t_0 F( r ) \std r \Bigr) %
v \rangle \\
\begin{aligned}
& = \langle u, x v \rangle \langle \evec{f( \cdot + s )_{t)}},
\evec{g( \cdot + s )_{t)}} \rangle \\
& \quad + \int^t_0 \langle u \evec{f( \cdot + s )_{r)}}, %
j_r \bigl( \phi^{\wh{f}(r + s)}_{\wh{g}(r + s)}( x ) \bigr) %
v \evec{g( \cdot + s )_{r)}} \langle \evec{f( \cdot + s)_{[ r, t )}},
\evec{g( \cdot + s)_{[ r , t )}} \rangle \std r \\
& = \langle u \evec{f( \cdot + s )_{t)}}, %
j_t( x ) v \evec{g( \cdot + s )_{t)}} \rangle \\
& = \langle u, \jt_t[ f( \cdot + s ), g( \cdot + s )]( x ) v %
\rangle.
\end{aligned}
\end{multline*}
Consequently,
\begin{align*}
\langle u, J_t( x ) v \rangle & = %
\langle u, \jt_s[ f, g ]( x ) v \rangle %
\langle \evec{f_{[ s, s + t )}}, \evec{g_{[ s, s + t )}} \rangle \\
 & \quad + \int^t_0 \langle u, %
\jt_s[ f, g ] \comp \jt_r[ f( \cdot + s ), g( \cdot + s ) ]\bigl( %
\phi^{\wh{f}( r + s )}_{\wh{g}( r + s )}( x ) \bigr) v \rangle %
\langle \evec{f_{[ s + r, s + t )}}, \evec{g_{[ s + r, s + t )}} %
\rangle \std r \\
& = \langle u, \jt_s[ f, g ]( x ) v \rangle %
\langle \evec{f_{[ s, s + t )}}, \evec{g_{[ s, s + t )}} \rangle \\
 & \quad + \int^t_0 \langle u, J_r\bigl( %
\phi^{\wh{f}(r + s )}_{\wh{g}(r + s )}( x ) \bigr) v \rangle %
\langle \evec{f_{[ s + r, s + t )}}, \evec{g_{[ s + r, s + t )}} %
\rangle \std r.
\end{align*}
On the other hand, by~\eqref{eqn:wqsde},
\begin{align*}
\langle u, \jt_{s + t}[ f, g ]( x ) v \rangle & = %
\langle u, x v \rangle %
\langle \evec{f_{s + t)}}, \evec{g_{ s + t)}} \rangle + %
\int^s_0 \langle u \evec{f_{s + t)}}, j_r\bigl( %
\phi^{\wh{f}( r )}_{\wh{g}( r )}( x ) \bigr) v \evec{g_{s + t)}} %
\rangle \std r \\
& \quad + \int^{s + t}_s \langle u \evec{f_{s + t)}}, j_r\bigl( %
\phi^{\wh{f}( r )}_{\wh{g}( r )}( x ) \bigr) v \evec{g_{s + t)}}
\rangle \std r \\
& = %
\langle u \evec{f_{s + t)}}, j_s( x ) v \evec{g_{s + t)}} \rangle + %
\int^t_0 \langle u \evec{f_{s + t)}}, j_{q + s}\bigl( %
\phi^{\wh{f}( q + s )}_{\wh{g}(q + s)}( x ) \bigr) v \evec{g_{s + t)}}
\rangle \std q \\
& = \langle u, \jt_s[ f, g ]( x ) v \rangle %
\langle \evec{f_{[ s, s + t )}}, \evec{g_{[ s, s + t )}} \rangle \\
& \quad + \int^t_0 \langle u, \jt_{q + s}[ f, g ]\bigl( %
\phi^{\wh{f}( q + s )}_{\wh{g}( q + s )}( x ) \bigr) v \rangle %
\langle \evec{f_{[ s + q, s + t )}}, \evec{g_{[ s + q, s + t )}}
\std q.
\end{align*}
Now set $K_t := J_t - \jt_{s+t}[ f, g ]$, so that
\[
\langle u, K_t( x ) v \rangle = \int^t_0 \langle u, %
K_r \bigl( \phi^{\wh{f}( r + s )}_{\wh{g}( r + s )}%
( x ) \bigr) v \rangle \, G( r ) \std r,
\]
where
$G : r \mapsto \langle %
\evec{f_{[ s + r, s + t )}}, \evec{g_{[ s + r, s + t )}} \rangle$
is continuous. As
\[
\| K_t \| \le 2 \exp \Bigl( \hlf \bigl( \| f \|^2 + \| g \|^2 \bigr)
\Bigr) \qquad \text{for all } t \ge 0,
\]
iterating the above and estimating as in the proof of
Lemma~\ref{lem:uniquesoln} shows that $K \equiv 0$. The density of
$\alg_0$ in $\alg$ and of continuous functions in $\elltwo$ now gives
(iv).

That $T$ is a semigroup follows from this cocycle property (iv): note
that
\[
T_{s + t} = \jt_{s + t}[ 0, 0 ] = %
\jt_s[ 0, 0 ] \comp \jt_t[ 0, 0 ] = T_s \comp T_t %
\qquad \text{for all } s, t \ge 0.
\]
Contractivity, complete positivity and strong continuity of $T$ are
immediate; the exponential identity holds because
\begin{equation}\label{eqn:sie}
\langle u, T_t( x ) v \rangle = \langle u, x v \rangle + %
\int_0^t \langle u, T_s\bigl( \phi_\vac^\vac( x ) \bigr) v \rangle
\std s
\end{equation}
for all $u$,~$v \in \ini$, $t \ge 0$ and $x \in \alg_0$,
by~\eqref{eqn:wqsde}. That $\alg_0$ is a core for the generator of $T$
follows from Lemma~\ref{lem:abe} and~\cite[Corollary~3.1.20]{BrR02a}.
\end{proof}

\begin{remark}
A $*$-homomorphic Feller cocycle as in Theorem~\ref{thm:cocycle} is
called a \emph{quantum flow}; a strongly continuous semigroup
$( T_t )_{t \ge 0}$ of completely positive contractions is known as a
\emph{quantum dynamical semigroup}, and the condition
$T_t( \id ) = \id$ for all $t \ge 0$ means that the semigroup is
\emph{conservative}; conservative quantum dynamical semigroups are
also known as \emph{quantum Markov semigroups}. Hence
Theorem~\ref{thm:cocycle} gives the existence of a quantum flow which
dilates a quantum Markov semigroup on the $C^*$~algebra $\alg$.
\end{remark}

\begin{remark}
By Theorem~\ref{thm:cocycle}, the component
$\phi_\omega^\omega = \tau$ of the flow generator $\phi$ is closable,
with $\overline{\tau}$ being the generator of the quantum Markov
semigroup $T$. However, closability of the bimodule map $\delta$ seems
to be a much more delicate issue and remains an open question.
\end{remark}

\begin{theorem}\label{thm:abelian}
Consider the family of $*$-homomorphisms $( \jt_t )_{t \ge 0}$
constructed in Theorems~\ref{thm:extension1}
and~\ref{thm:extension2}. If $\alg_c$ is a commutative $*$-subalgebra
of $\alg$ such that
\[
\begin{array}{rl}
\text{\rm (i)} & %
\phi( \alg_c \cap \alg_0 ) \subseteq \alg_c \otimes \bopp \\[1ex]
\text{and} \quad \text{\rm (ii)} & %
\alg_c \cap \alg_0 \text{ is dense in } \alg_c
\end{array}
\]
then the family $\{ \jt_t( a ) : t \ge 0, \ a \in \alg_c \}$ is
commutative, \ie the commutator
$[ \jt_s( a ), \jt_t( b ) ] = 0$ for all $s$, $t \ge 0$ and
$a$,~$b \in \alg_c$.
\end{theorem}
\begin{proof}
The result is immediate when $s = t$, so assume without loss of
generality that $s < t$ and let $b \in \alg_c \cap \alg_0$; if
\[
K_t( b ) := %
\langle u \evec{f}, [\jt_s (a), \jt_t (b)] v \evec{g} \rangle = 0,
\]
where $u$, $v \in \ini$, $f$, $g \in \elltwo$ and $a \in \alg_c$ are
arbitrary, then the result follows by~(ii) and the continuity
of~$\jt_t$.

Write $\jt_t( b ) = \jt_s( b ) + \int_s^t L_r \std \Lambda_r$, where
$L = ( L_r )_{r \ge 0}$ is the process defined
in~\eqref{eqn:integrand} with $x$ changed to $b$. It is
straightforward, using adaptedness, to show that
\[
A \int^t_s L_r \std \Lambda_r B = %
\int^t_s \Sigma( A \otimes 1_\mmul ) \, L_r \, %
\Sigma( B \otimes 1_\mmul ) \std \Lambda_r
\]
for any
$A$, $B \in \bop{\ini \ctimes \fock_{[ 0, s )}}{} \ctimes %
1_{\fock_{[ s, \infty )}}$, where $\Sigma$ is the swap isomorphism
defined after~\eqref{eqn:integrand}. Since $\jt$ is a strong solution
of the QSDE~\eqref{eqn:EHqsde}, by Corollary~\ref{cor:stsoln}, it
follows that
\[
K_t( b ) = \int^t_s %
K_r\bigl( \phi^{\wh{f}( r )}_{\wh{g}( r )}( b ) \bigr) \std r.
\]
Assumption~(i) allows us to iterate this identity; noting also that
\[
| K_r( c ) | \le 2 \| u \| \, \| v \| \, %
\| \evec{f} \| \, \| \evec{g} \| \, \| a \| \, \| c \| %
\qquad \text{for all } c \in \alg_c \cap \alg_0,
\]
one readily obtains the estimate
\[
| K_t( b ) | \le 2 \| u \| \, \| v \| \, \| \evec{f} \| \, %
\| \evec{g} \| \, \| a \| C_b M_b^n \frac{1}{n!} %
\Bigl( %
\int_s^t \| \wh{f}( r ) \| \, \| \wh{g}( r ) \| \std r \Bigr)^n,
\]
where $C_b$ and $M_b$ are constants associated to $b$ through its
membership of $\alg_\phi$. Letting $n \to \infty$ gives the result.
\end{proof}

\begin{remark}
If $\alg$ is commutative then conditions~(i) and~(ii) of
Theorem~\ref{thm:abelian} are satisfied automatically when
$\alg_c = \alg$, so Theorems~\ref{thm:extension1} and
\ref{thm:extension2} produce classical Markov semigroups in this case.
However, Theorem~\ref{thm:abelian} also allows for the possibility of
dealing with different commutative subalgebras that do not commute
with one another, a necessary feature of quantum dynamics.
\end{remark}

\section{Random walks on groups}\label{sec:walks}

\begin{definition}\label{dfn:walks}
Let
$\alg = C_0( G ) \oplus \C 1 \subseteq %
\bopp\bigl( \ell^2( G ) \bigr)$,
where $G$ is a discrete group and $x \in C_0( G )$ acts
on~$\ell^2( G )$ by multiplication, and let
$\alg_0 = \lin\{ 1, \, e_g : g \in G \}$, where
$e_g( h ) := \tfn{g = h}$ for all $h \in G$. That is, $\alg$ is the
unitisation of the $C^*$~algebra of functions on $G$ which vanish at
infinity and $\alg_0$ is the dense unital subalgebra generated by the
functions with finite support; as positivity in the $C^*$-algebraic
sense corresponds here to the pointwise positivity of
functions,~$\alg_0$ contains its square roots.

Let $H$ be a non-empty finite subset of $G \setminus \{ e \}$ and let
the Hilbert space~$\mul$ have orthonormal basis $\{ f_h : h \in H \}$;
the maps
\[
\lambda_h : G \to G; \ g \mapsto h g \qquad ( h \in H )
\]
correspond to the permitted moves in the random walk constructed
on~$G$.
\end{definition}

\begin{lemma}\label{lem:walks}
Given a \emph{transition function}
\[
t : H \times G \to \C; \ ( h, g ) \mapsto t_h( g ),
\]
the map
\[
\phi : \alg_0 \to \alg_0 \otimes \bopp; \ %
x \mapsto \begin{bmatrix}
 \sum_{h \in H} | t_h |^2 ( x \comp \lambda_h - x ) & %
\sum_{h \in H} \overline{t_h} ( x \comp \lambda_h - x ) \otimes %
\bra{f_h} \\[1ex]
 \sum_{h \in H} t_h ( x \comp \lambda_h - x ) \otimes \ket{f_h} & %
\sum_{h \in H} ( x \comp \lambda_h - x ) \otimes \dyad{f_h}{f_h}
\end{bmatrix}
\]
is a flow generator such that
\[
\phi( e_g ) = e_g \otimes m_e( g ) + \sum_{h \in H} e_{h^{-1}g} %
\otimes m_h( h^{-1} g ) \qquad \text{for all } g \in G,
\]
where
\[
m_e( g ) := \begin{bmatrix}
 -\sum_{h \in H} | t_h( g ) |^2 & %
-\sum_{h \in H} \overline{t_h( g )} \bra{f_h} \\[1ex]
 -\sum_{h \in H} t_h( g ) \ket{f_h} & -1_\mul
\end{bmatrix} \quad \text{and} \quad %
m_h( g ) := \begin{bmatrix}
 | t_h( g ) |^2 & %
\overline{t_h( g )} \bra{f_h} \\[1ex]
 t_h( g ) \ket{f_h} & \dyad{f_h}{f_h}
\end{bmatrix}.
\]
Hence
\[
\phi_n( e_g ) = %
\sum_{h_1 \in H \cup \{ e \}} \cdots \sum_{h_n \in H \cup \{ e \}} %
e_{h_n^{-1} \cdots h_1^{-1} g} \otimes %
m_{h_n}( h_n^{-1} \cdots h_1^{-1} g ) \otimes \cdots %
\otimes m_{h_1}( h_1^{-1} g )
\]
for all $n \in \N$ and $g \in G$.
\end{lemma}
\begin{proof}
The first claim is readily verified with the aid of
Lemma~\ref{lem:gen}; the second is immediate.
\end{proof}

\begin{theorem}\label{thm:walks}
Let $\alg$ be as in Definition~\ref{dfn:walks} and $\phi$ as in
Lemma~\ref{lem:walks}. If the transition function~$t$ is chosen such
that $\alg_\phi = \alg_0$ then there exists an adapted family of
unital $*$-homomorphisms
$\bigl( \jt_t : \alg \to \bop{\ini \ctimes \fock}{} \bigr)_{t \ge 0}$
which forms a Feller cocycle in the sense of Theorem~\ref{thm:cocycle}
and satisfies the quantum stochastic differential
equation~\eqref{eqn:EHqsde} in the strong sense on~$\alg_0$ for
all~$t \ge 0$. Setting
\[
T_t( a ) := %
\bigl( \id_\ini \otimes \bra{\evec{0}} \bigr) \jt_t( a ) %
\bigl( \id_\ini \otimes \ket{\evec{0}} \bigr) %
\qquad \text{for all } a \in \alg \text{ and } t \ge 0
\]
gives a classical Markov semigroup~$T$ on $\alg$ whose generator is
the closure of
\[
\tau : \alg_0 \to \alg_0; \ %
x \mapsto \sum_{h \in H} | t_h |^2 ( x \comp \lambda_h - x ).
\]
\end{theorem}
\begin{proof}
This follows from Theorems~\ref{thm:extension1} and~\ref{thm:abelian}
together with Lemma~\ref{lem:walks}.
\end{proof}

\begin{remark}\label{rem:walks}
Given $g \in G$, let
$A := \bigl[ B \ 1_\mul \bigr] \in %
\bop{\C \oplus \mul}{\mul}$,
where $B := \sum_{h \in H} t_h( g ) \ket{f_h}$. Then
$m_e( g ) = -A^* A$ and
\[
\| m_e( g ) \| = \| A A^* \| = \| B B^* + 1_\mul \| %
= \| B^* B \| + 1 = 1 + \sum_{h \in H} | t_h( g ) |^2.
\]
It may be shown similarly that
$\| m_h( g ) \| = 1 + | t_h( g ) |^2$ for all $g \in G$ and $h \in H$,
so if
\begin{equation}\label{eqn:walks}
M_g := \lim_{n \to \infty} %
\sup\bigl\{ | t_h( h_n^{-1} \cdots h_1^{-1} g ) | : %
h_1, \ldots, h_n \in H \cup \{ e \}, \ h \in H \bigr\} < \infty
\end{equation}
then
\[
\| \phi_n( e_g ) \| \le ( 1 + | H | + 2 | H | M_g^2 )^n %
\qquad \text{for all } n \in \Z_+,
\]
where $| H |$ denotes the cardinality of $H$. Hence
$\alg_\phi = \alg_0$ if \eqref{eqn:walks} holds for all $g \in G$.
\end{remark}

\begin{remark}
If $t$ is bounded then clearly~\eqref{eqn:walks} holds for all
$g \in G$. In this case, there exist bounded operators
$L \in \bop{\ini}{\ini \ctimes \mul}$,
$S \in \bop{\ini \ctimes \mul}{}$ and
$F \in \bop{\ini \ctimes \mmul}{}$ such that
\[
L = \sum_{h \in H} t_h \otimes \ket{f_h}, \quad %
S = \sum_{h \in H} S_h \otimes \dyad{f_h}{f_h} %
\quad \text{and} \quad %
F = \begin{bmatrix}
 -\frac{1}{2} L^* L & -L^* \\[1ex]
 S L & S - 1_{\ini \otimes \mul}
\end{bmatrix},
\]
where $t_h$ acts by multiplication and $S_h$ is the unitary operator
on $\ell^2( G )$ such that $e_g \mapsto e_{h g}$.

It follows from \cite[Theorems~7.1 and 7.5]{LiW00b} that the
Hudson--Parthasarathy QSDE
\[
U_0 = I_{\ini \otimes \fock}, \qquad %
\rd U_t = %
(F \otimes 1_\fock) \Sigma ( U_t \otimes I_\mmul ) \std \Lambda_t,
\]
where $\Sigma$ is the swap isomorphism defined
after~\eqref{eqn:integrand}, has a unique solution which is a unitary
cocycle. Furthermore, by \cite[Theorem~7.4]{LiW00b}, setting
\[
k_t( a ) := U^*_t ( a \otimes 1_\fock ) U_t %
\quad \text{for all } a \in \bop{\ini}{} \text{ and } t \ge 0
\]
defines a quantum flow $k$ with generator
\[
\varphi : \bop{\ini}{} \to \bop{\ini \ctimes \mmul}{}; \ %
a \mapsto %
( a \otimes 1_\mmul ) F + F^* ( a \otimes 1_\mmul ) + %
F^* \Delta ( a \otimes 1_\mmul ) F.
\]
A short calculation shows that $\varphi$ is of the form covered by
Lemma~\ref{lem:bddphi}, with
\[
\pi( a ) = S^* ( a \otimes 1_\mul ) S, \qquad \delta( a ) = %
-L a + \pi( a ) L \quad \text{and} \quad \tau( a ) = -\hlf\{ L^* L, a
\} + L^* \pi( a ) L
\]
for all $a \in \bop{\ini}{}$. It follows that
$\varphi|_{\alg_0} = \phi$, where $\phi$ is the flow generator of
Lemma~\ref{lem:walks}, and so the cocycle $\jt$ given by
Theorem~\ref{thm:walks} is the restriction of $k$ to $\alg$. However,
this construction by conjugation does not give the Feller property,
that $\alg$ is preserved by $k$.
\end{remark}

\begin{example}
If $G = ( \Z, + )$, $H = \{ \pm 1 \}$ and the transition function $t$
is bounded, with $t_{+1}( g ) = 0$ for all $g < 0$ and
$t_{-1}( g ) = 0$ for all $g \le 0$, then the Markov semigroup $T$
given by Theorem~\ref{thm:walks} corresponds to the classical
birth-death process with birth and dates rates $| t_{+1} |^2$ and
$| t_{-1} |^2$, respectively. The cocycle constructed here is
Feller, as it acts on $\alg = C_0( \Z ) \oplus \C 1$, in contrast to
\cite[Example~3.3]{PaS90}, where the cocycle acts on the whole of
$\ell^\infty( \Z )$.
\end{example}

\begin{remark}
If $G = ( \Z, + )$, $H = \{ +1 \}$ and $t_{+1} : g \mapsto 2^g$ then
$M_g = 2^g$ and the condition \eqref{eqn:walks} holds for all
$g \in G$. Thus Theorem~\ref{thm:walks} applies to examples where the
transition function $t$ is unbounded.
\end{remark}

\section{The symmetric quantum exclusion process}\label{sec:sqep}

This section was inspired by Rebolledo's treatment of the quantum
exclusion process: see~\cite[Examples~2.4.3 and~4.1.3]{Reb05}.

\begin{definition}
Let $I$ be a non-empty set. The \emph{CAR algebra} is the unital
$C^*$~algebra $\alg$ with generators $\{ b_i : i \in I \}$, subject to
the anti-commutation relations
\begin{equation}\label{eqn:car}
\{ b_i, b_j \} = 0 \qquad \text{and} \qquad %
\{ b_i, b_j^* \} = \tfn{i = j} \qquad %
\text{for all } i, j \in I.
\end{equation}
It follows from~\eqref{eqn:car} that the $b_i$ are nonzero partial
isometries for all $i \in I$.

As is well known \cite[Proposition~5.2.2]{BrR02b}, $\alg$ is
represented faithfully and irreducibly on
$\fock_-\bigl( \ell^2( I ) \bigr)$, the Fermionic Fock space over
$\ell^2( I )$; in other words, we may (and do) suppose that
$\alg \subseteq \bop{\ini}{}$, where
$\ini := \fock_-\bigl( \ell^2 ( I ) \bigr)$, and the algebra identity
$\id = \id_\ini$.
\end{definition}

\begin{remark}
The elements of $I$ may be taken to correspond to sites at which
Fermionic particles may exist, with the operators $b_i$ and $b_i^*$
representing the annihilation and creation, respectively, of a
particle at site~$i$.
\end{remark}

\begin{notation}
Let $\alg_0$ be the unital algebra generated by
$\{ b_i, b_i^* : i \in I \}$; by definition, this is a norm-dense
unital $*$-subalgebra of $\alg$.
\end{notation}

\begin{lemma}\label{lem:fdsa}
For each $x \in \alg_0$ there exists a finite subset
$J \subseteq I$ such that $x$ lies in the finite-dimensional
$*$-subalgebra
\[
\alg_J := %
\lin\bigl\{ b_{j_1}^* \cdots b_{j_q}^* b_{i_1} \cdots b_{i_p} : %
0 \le p, q \le |J|, \ \{ i_1, \ldots, i_p \} \in
J^{(p)}, \ \{ j_1, \ldots, j_q \} \in J^{(q)} \} \subseteq \alg_0,
\]
where $J^{(p)}$ denote the set of subsets of~$J$ with cardinality $p$
\etc. Consequently, $\alg$ is an AF algebra and $\alg_0$ contains its
square roots.
\end{lemma}
\begin{proof}
By employing the anti-commutation relations~\eqref{eqn:car}, any
finite product of terms from the generating set
$\{ b_i, b_i^* : i \in I \}$ may be reduced to a linear combination of
words of the form
\begin{equation}\label{eqn:canonical}
b_{j_1}^* \cdots b_{j_q}^* b_{i_1} \cdots b_{i_p},
\end{equation}
where $i_1$, \ldots, $i_p$ are distinct elements of $I$, as are $j_1$,
\ldots, $j_q$, and $p$,~$q \in \Z_+$, with an empty product equal
to~$\id$. As every element of $\alg_0$ is a finite linear combination
of such terms, the first claim follows. The second claim holds by
Remark~\ref{rmk:af}.
\end{proof}

\begin{definition}
Let $\{ \alpha_{i, j} : i, j \in I \} \subseteq \C$ be a fixed
collection of \emph{amplitudes}. We may view
$( I, \{ \alpha_{i, j} : i, j \in I \} )$ as a weighted directed
graph, where $I$ is the set of vertices, an edge exists from $i$
to~$j$ if $\alpha_{i, j} \neq 0$ and $\alpha_{i, j}$ is a complex
weight on the edge from vertex~$i$ to vertex~$j$, which may differ
from the weight $\alpha_{j, i}$ from $j$ to $i$.

For all $i \in I$, let
\[
\supp( i ) := \{ j \in I : \alpha_{i, j} \neq 0 \}
\ \text{ and } \ %
\supp^+( i ) := \supp( i ) \cup \{ i \}.
\]
Thus $\supp( i )$ is the set of sites with which site $i$ interacts
and $| \supp( i ) |$ is the valency of the vertex $i$. We require that
the valencies are finite:
\begin{equation}\label{eqn:finite}
| \supp( i ) | < \infty
\quad \text{for all } i \in I.
\end{equation}

The transport of a particle from site~$i$ to site~$j$ with amplitude
$\alpha_{i, j}$ is described by the operator
\[
t_{i, j} :=\alpha_{i, j} \, b_j^* b_i.
\]
\end{definition}

\begin{definition}
Let $\{ \eta_i : i \in I \} \subseteq \R$ be fixed. The total energy
in the system is given by
\[
h := \sum_{i \in I} \eta_i \, b_i^* b_i,
\]
where $\eta_i$ gives the energy of a particle at site~$i$. If
the set $\{ i \in I: \eta_i \neq 0 \}$ is infinite then the proper
interpretation of~$h$ involves issues of convergence; below it
will only appear in a commutator with elements of $\alg_0$, which is
sufficient to give a well-defined quantity.
\end{definition}

\begin{lemma}\label{lem:tau}
Let
\[
\tau_{i, j}( x ) := %
t_{i, j}^* [ t_{i, j}, x ] + [ x, t_{i, j}^* ] t_{i, j} = %
| \alpha_{i, j} |^2 \bigl( b_i^* b_j [ b_j^* b_i, x ] + %
[ x, b_i^* b_j ] b_j^* b_i \bigr)
\]
for all $i$,~$j \in I$ and $x \in \alg$, and let
\begin{equation}\label{eqn:energy}
[ h, x ] := \sum_{i \in I } \eta_i [ b_i^* b_i, x ]
\end{equation}
for all $x \in \alg_0$. Setting
\begin{equation}
\tau( x ) := %
\I [ h, x ] - \hlf \sum_{i, j \in I} \tau_{i, j}( x )
\end{equation}
defines a $*$-linear map $\tau : \alg_0 \to \alg_0$.
\end{lemma}
\begin{proof}
Let $x \in \alg_0$ and note that $x \in \alg_J$ for some finite set
$J \subseteq I$, by Lemma~\ref{lem:fdsa}. Furthermore,
\[
[ b_j^* b_i, x ] = b^*_j \{ b_i, x \} - \{ b^*_j, x \} b_i = 0 %
\qquad \text{whenever } i \not \in J \text{ and } j \not \in J,
\]
so
\[
[ h, x ] = %
\sum_{i \in J} \eta_i [ b_i^* b_i , x ] \in \alg_J %
\qquad \text{and} \qquad %
\tau( x ) = \I [ h, x ] - %
\hlf \sum_{i, j \in J^+} \tau_{i, j}( x ) \in %
\alg_{J^{+}},
\]
where
\begin{equation}\label{eqn:Jdash}
J^+ := \bigcup_{k \in J} \supp^+( k ).
\end{equation}
Hence $\tau( \alg_J ) \subseteq \alg_{J^+}$ and, as~\eqref{eqn:finite}
implies that $J^+$ is finite, it follows that $\alg_0$ is invariant
under~$\tau$. The $*$-linearity of $\tau$ is immediately verified.
\end{proof}

\begin{lemma}\label{lem:delta}
Let
\[
\delta_{i, j}( x ) := %
[ t_{i, j} , x ] = \alpha_{i, j} ( b_j^* b_i x - x b_j^* b_i )
\]
for all $i$, $j \in I$ and $x \in \alg$, and let $\mul$ be a Hilbert
space with orthonormal basis $\{ f_{i, j} : i, j \in I \}$. Setting
\begin{equation}\label{eqn:deldef}
\delta( x ) := %
\sum_{i, j \in I} \delta_{i, j}( x ) \otimes \ket{f_{i, j}}
\end{equation}
for all $x \in \alg_0$ defines a linear map $\delta : \alg_0 \to
\alg_0 \otimes \ket{\mul}$ such that
\begin{align}
\delta( x y ) & = %
\delta( x ) y + ( x \otimes \id_\mul ) \delta( y ) %
\label{eqn:derv} \\[1ex]
\text{and} \qquad \delta^\dagger( x ) \delta( y ) & = %
\tau( x y ) - \tau( x ) y - x \tau( y )
\end{align}
for all $x$, $y \in \alg_0$, where $\tau$ is as defined in
Lemma~\ref{lem:tau}.
\end{lemma}
\begin{proof}
The series in~\eqref{eqn:deldef} contains only finitely many
terms, since if $x \in \alg_J$ then
\[
\delta_{i, j}( x ) = 0 \qquad
\text{when } \{ i, j \} \not \subseteq J^+.
\]
Hence $\delta$ is well defined, and~\eqref{eqn:derv} holds because
each $\delta_{i, j}$ is a derivation. A short calculation shows that
\begin{equation}\label{eqn:diss}
\tau_{i, j}( x y ) - \tau_{i, j}( x ) y - x \tau_{i, j}( y ) = %
-2 \delta_{i, j}^\dagger( x ) \delta_{i, j}( y )
\end{equation}
for all $x$, $y \in \alg$. Since $x \mapsto [ b_i^* b_i, x ]$ is a
derivation for all $i \in I$, it follows from~\eqref{eqn:diss} that
\[
\tau( x y ) - \tau( x ) y - x \tau( y ) = %
\sum_{i, j \in I} \delta_{i, j}^\dagger( x ) \delta_{i, j}( y ) = %
\delta^\dagger( x ) \delta( y ) %
\qquad \text{for all } x, y \in \alg_0. %
\qed
\]
\renewcommand{\qed}{}
\end{proof}

\begin{lemma}\label{lem:sqepphi}
The map
\begin{equation}\label{eqn:sqepphi}
\phi : \alg_0 \to \alg_0 \otimes \bopp; \ %
x \mapsto \begin{bmatrix}
 \tau( x ) & \delta^\dagger( x ) \\[1ex]
 \delta( x ) & 0
\end{bmatrix},
\end{equation}
where $\tau$, $\delta$ and $\delta^\dagger$ are as defined in
Lemmas~\ref{lem:tau} and~\ref{lem:delta}, is a flow generator.

If the amplitudes satisfy the \emph{symmetry condition}
\begin{equation}\label{eqn:symmetric}
| \alpha_{i, j} | = | \alpha_{j, i} | \qquad \text{for all } i, j \in I
\end{equation}
then, for all $n \in \N$ and $i_0 \in I$,
\begin{equation}\label{eqn:phid}
\phi_n( b_{i_0} ) = \sum_{i_1 \in \supp^+( i_0 )} \cdots %
\sum_{i_n \in \supp^+( i_{n - 1} )} b_{i_n} \otimes %
B_{i_{n - 1}, i_n} \otimes \cdots \otimes B_{i_0, i_1},
\end{equation}
where
\[
B_{i, j} := \tfn{j = i} \lambda_i \dyad{\vac}{\vac} + %
\dyad{\vac}{\alpha_{i, j} f_{i, j}} - %
\dyad{\alpha_{j, i} f_{j, i}}{\vac}
\]
and
\[
\lambda_i := %
-\I \eta_i - \hlf \sum_{j \in \supp( i )} | \alpha_{j, i} |^2
\]
for all $i$, $j \in I$.
\end{lemma}
\begin{proof}
The first claim is an immediate consequence of Lemmas~\ref{lem:tau},
\ref{lem:delta} and \ref{lem:gen}.

If $i$,~$j$,~$k \in I$ then a short calculation shows that
\[
\tau_{j, k}( b_i ) = \left\{ \begin{array}{ll}
 | \alpha_{i, i} |^2 b_i & ( j = i, k = i ), \\[1ex]
 | \alpha_{j, i} |^2 b_j^* b_j b_i & ( j \neq i, k = i ), \\[1ex]
 | \alpha_{i, k} |^2 b_k b_k^* b_i & ( j = i, k \neq i ), \\[1ex]
 0 & ( j \neq i, k \neq i ).
\end{array} \right.
\]
Since
\[
[ h, b_i ] = \sum_{j \in I} \eta_j [ b_j^* b_j, b_i ] = %
\eta_i [ b_i^* b_i, b_i ] = - \eta_i b_i,
\]
the symmetry condition~\eqref{eqn:symmetric} implies that
\[
\tau( b_i ) = \lambda_i b_i \qquad \text{for all } i \in I.
\]
Furthermore, if $i$, $j$, $k \in I$ then
\[
\delta_{j, k}( b_i ) = %
\alpha_{j, k} ( b_k^* b_j b_i - b_i b_k^* b_j ) = %
-\alpha_{j, k} \{ b_k^*, b_i \} b_j = %
- \tfn{k = i} \alpha_{j, i} \, b_j
\]
and
\[
\delta^\dagger_{j, k}( b_i ) = %
\overline{\alpha_{j, k}} ( b_i b_j^* b_k - b_j^* b_k b_i ) = %
\overline{\alpha_{j, k}} \{ b_i, b_j^* \} b_k = %
\tfn{j = i} \overline{\alpha_{i, k}} \, b_k;
\]
thus
\[
\delta( b_i ) = %
\sum_{j, k \in I} \delta_{j, k}( b_i ) \otimes \ket{f_{j, k}} = %
-\sum_{j \in \supp( i )} \alpha_{j, i} \, b_j \otimes \ket{f_{j, i}}
\]
and
\[
\delta^\dagger( b_i ) = %
\sum_{j, k \in I} \delta^\dagger_{j, k}( b_i ) \otimes \bra{f_{j, k}}
= \sum_{k \in \supp( i )} %
\overline{\alpha_{i, k}} \, b_k \otimes \bra{f_{i, k}}.
\]
Hence
\begin{align*}
\phi( b_i ) & = \lambda_i b_i \otimes \dyad{\vac}{\vac} - %
\sum_{j \in \supp( i )} %
\alpha_{j, i} b_j \otimes \dyad{f_{j, i}}{\vac} + %
\sum_{k \in \supp( i )} %
\overline{\alpha_{i, k}} b_k \otimes \dyad{\vac}{f_{i, k}} \\[1ex]
 & = \sum_{j \in \supp^+( i )} b_j \otimes %
\bigl( \tfn{j = i} \lambda_i \dyad{\vac}{\vac} + %
\dyad{\vac}{\alpha_{i, j} f_{i, j}} - %
\dyad{\alpha_{j, i} f_{j, i}}{\vac} \bigr)
\end{align*}
and the identity~\eqref{eqn:phid} follows.
\end{proof}

\begin{theorem}
Let $\alg$ be the CAR algebra and let $\phi$ be defined as in
Lemma~\ref{lem:sqepphi}. If the amplitudes $\{ \alpha_{i , j} \}$
and energies $\{ \eta_i \}$ are chosen so that $\alg_\phi = \alg_0$
then there exists an adapted family of unital $*$-homomorphisms
$\bigl( j_t : \alg \to \bop{\ini \ctimes \fock}{} \bigr)_{t \ge 0}$
which forms a Feller cocycle in the sense of Theorem~\ref{thm:cocycle}
and satisfies the quantum stochastic differential
equation~\eqref{eqn:EHqsde} in the strong sense on~$\alg_0$ for
all~$t \ge 0$. Setting
\[
T_t( a ) := %
\bigl( \id_\ini \otimes \bra{\evec{0}} \bigr) j_t( a ) %
\bigl( \id_\ini \otimes \ket{\evec{0}} \bigr) %
\qquad \text{for all } a \in \alg \text{ and } t \ge 0
\]
gives a quantum Markov semigroup~$T$ on $\alg$ whose generator is the
closure of
\[
\tau : \alg_0 \to \alg_0; \ %
x \mapsto \I \sum_{i \in I} \eta_i [ b_i^* b_i , x ] - %
\hlf \sum_{i, j \in I} | \alpha_{i, j} |^2 \bigl( %
b_i^* b_j [ b_j^* b_i, x ] + [ x, b_i^* b_j ] b_j^* b_i \bigr).
\]
\end{theorem}
\begin{proof}
This is an immediate consequence of Theorem~\ref{thm:extension1},
Theorem~\ref{thm:cocycle} and Lemma~\ref{lem:sqepphi}.
\end{proof}

\begin{example}\label{ex:sym+bounds}
Suppose that the amplitudes satisfy the symmetry
condition~\eqref{eqn:symmetric}, and further that there are
uniform bounds on the amplitudes, valencies and energies:
\begin{equation}\label{eqn:allbdd}
M := \sup_{ i, j \in I } | \alpha_{ i, j } | < \infty,
\qquad
V := \sup_{i \in I} | \supp( i ) | < \infty
\qquad \text{and} \qquad
H := \sup_{i \in I} | \eta_i | < \infty.
\end{equation}
It follows that
\[
| \lambda_i | \le | \eta_i | + \hlf V M^2 %
\qquad \text{and} \qquad %
\| B_{i, j} \| \le %
| \lambda_i | + 2 M \le H + \hlf V M^2 + 2 M
\]
for all $i$,~$j \in I$. Hence, for all $n \in \Z_+$,
\[
\| \phi_n( b_i ) \| \le %
( V + 1 )^n \bigl( H + \hlf V M^2 + 2 M \bigr)^n
\]
and so $\alg_\phi = \alg_0$, by Corollary~\ref{cor:bound}. Hence
there is a flow on $\alg$ for this generator.
\end{example}

\begin{example}
We can lift the boundedness assumptions in Example~\ref{ex:sym+bounds}
by taking $I$ to be a disjoint union of subsets,
\[
I = \bigsqcup_{k \in K} I_k,
\]
such that there is no transport between any of these subsets, \ie
\[
\alpha_{ i, j } \neq 0 \text{ only if there is some } %
k \in K \text{ such that } i, j \in I_k.
\]
Assume the symmetry condition~\eqref{eqn:symmetric} once
again. Suppose that in each $I_k$ the conditions of~\eqref{eqn:allbdd}
are satisfied, but with respect to constants $M_k$, $V_k$ and
$H_k$ that depend on $k$. Then, if $i \in I_k$, we get the estimate
\[
\| \phi_n( b_i ) \| \le %
( V_k + 1 )^n \bigl( H_k + \hlf V_k M_k^2 + 2 M_k \bigr)^n
\]
and so $\alg_\phi = \alg_0$ once more, but now it is possible that
$M = \infty$ \etc.
\end{example}

\begin{example}
To create an example where the graph associated to $I$ has only one
component, but where we do not assume $M < \infty$ as in
Example~\ref{ex:sym+bounds}, assume once again that
$I$ is decomposed into a disjoint union:
\[
I = \bigsqcup_{k \in \Z_+} I_k \qquad %
\text{with } | I_k | < \infty \text{ for all } k \in \Z_+.
\]
This time assume, as well as the symmetry
condition~\eqref{eqn:symmetric}, that $\alpha_{i,j} = 0$ unless there
is some $k \in \Z_+$ such that $i \in I_k$ and $j \in I_{k + 1}$, or
$j \in I_k$ and $i \in I_{k + 1}$, so that there is transport only
between neighbouring levels in $I$. Set
\[
a_k = \sup \{ |\alpha_{ i, j }|: i \in I_k, \ j \in I_{k + 1} \} %
\qquad \text{for all } k \in \Z_+,
\]
and furthermore assume that the energies are bounded,
\ie $H < \infty$.

Now if $k \in \N$ and $i \in I_k$ then
\begin{align*}
\sum_{j \in \supp^+( i )} \| B_{i, j} \| %
 & \le \| B_{i, i} \| + \sum_{j \in I_{k - 1}} \| B_{i, j} \| + %
\sum_{j \in I_{k + 1}} \| B_{i, j} \| \\ %
& \le | \lambda_i | + 2 | I_{k - 1} | a_{k - 1} + %
2 | I_{k + 1} | a_k,
\end{align*}
with a similar estimate holding if $i \in I_0$. Furthermore,
\[
| \lambda_i | \le H + %
\hlf | I_{k - 1} | a_{k - 1}^2 + %
\hlf | I_{k + 1} | a_k^2.
\]
As in Example~\ref{ex:sym+bounds}, if it can be shown that
\[
\sum_{j \in \supp^+( i )} \| B_{ i, j } \| \le C
\]
for some constant $C$ that does not depend on $i$, it follows that
$\| \phi_n( b_i ) \| \le C^n$ for each $n \in \Z_+$ and $i \in I$, and
so $\alg_\phi = \alg_0$ once more. Here, the previous working shows
this will hold if there are constants $a > 0$, $b > 0$ and $p \ge 1$
such that
\[
a_k \le \frac{a}{( k + 2 )^p} \quad \text{and} \quad %
| I_k | \le b ( k + 1 )^p \qquad \text{for all } k \in \Z_+.
\]
It is clear that this can yield an example where $M = \infty$, \ie
there is no upper bound on the valencies.
\end{example}

\section[Flows on universal C* algebras]%
{Flows on universal \boldmath{$C^*$}~algebras}\label{sec:uni}

\subsection{The non-commutative torus}

\begin{definition}\label{dfn:nct}
Let $\lambda \in \T$, the set of complex numbers with unit
modulus. The \emph{non-commutative torus} is the universal
$C^*$~algebra $\alg$ generated by unitaries $U$ and $V$ which satisfy
the relation
\[
U V = \lambda V U.
\]
Let $\alg_0$ denote the dense $*$-subalgebra of $\alg$ generated by
$U$ and $V$.

There is a faithful trace $\tr$ on~$\alg$ such that
$\tau( U^m V^n ) = \tfn{m = n = 0}$ for all $m$, $n \in \Z$; the proof
of this in \cite[pp.166--168]{Dav96} is valid for all $\lambda$.
Consequently $\{U^m V^n : m, n \in \Z\}$ is a basis for~$\alg_0$.
\end{definition}

\begin{lemma}\label{lem:nct-real}
Let $\ini := \ell^2( \Z^2 )$, let
\[
( U_c u )_{m, n} = u_{m + 1, n} \quad \text{and} \quad %
( V_c u )_{m, n} = \lambda^m u_{m, n + 1} \qquad %
\text{for all } u \in \ini \text{ and } m, n \in \Z,
\]
and let $\alg_c \subseteq \bop{\ini}{}$ be the $C^*$~algebra generated
by $U_c$ and $V_c$. There is a $C^*$~isomorphism from~$\alg$ to
$\alg_c$ such that $U \mapsto U_c$ and $V \mapsto V_c$. Moreover,
under this map the trace $\tr$ corresponds to the vector state given
by $e \in \ini$ such that $e_{m, n} = \tfn{m = n = 0}$ for all $m$,
$n \in \Z$.
\end{lemma}
\begin{proof}
Unitarity of $U_c$ and $V_c$ is immediately verified, as is the
identity $U_c V_c = \lambda V_c U_c$, so the universality of~$\alg$
gives a surjective $*$-homomorphism from $\alg$ to $\alg_c$.
Injectivity is a consequence of the final observation, that $\tr$
corresponds to the vector state given by~$e$.
\end{proof}

From now on we will identify $\alg$ and $\alg_c$.

\begin{definition}\label{dfn:rotates}
For each $(\mu, \nu) \in \T^2$, let $\pi_{\mu, \nu}$ be the
automorphism of $\alg$ such that
\[
\pi_{\mu, \nu}( U^m V^n ) = \mu^m \nu^n U^m V^n \qquad %
\text{for all } m, n \in \Z;
\]
the existence of $\pi_{\mu, \nu}$ is an immediate consequence of
universality.
\end{definition}

The proofs of the next two lemmas are a matter of routine algebraic
computation.

\begin{lemma}\label{lem:newderivs}
For all $a$, $b \in \Z$, define maps ${}_a\delta : \alg_0 \to \alg_0$
and $\delta_b : \alg_0 \to \alg_0$ by linear extension of the
identities
\[
{}_a\delta( U^m V^n ) = m U^{a + m} V^n %
\quad \text{and} \quad %
\delta_b( U^m V^n ) = n \lambda^{-b m} U^m V^{b + n} %
\qquad \text{for all } m, n \in \Z.
\]
Then ${}_a\delta$ is a $\pi_{1, \lambda^a}$-derivation and $\delta_b$
is a $\pi_{\lambda^{-b}, 1}$-derivation; moreover, their adjoints are
such that
\[
{}_a\delta^\dagger( U^m V^n ) = -m \lambda^{a n} U^{-a + m} V^n %
\quad \text{and} \quad %
\delta_b^\dagger( U^m V^n ) = -n U^m V^{-b + n}
\]
for all $m$, $n \in \Z$.
\end{lemma}

\begin{remark}
The sufficient condition in Lemma~\ref{lem:newderivs} is also
necessary. It is easy to show that if~${}_a\delta$ is a
$\pi_{\mu, \nu}$-derivation then $\mu = 1$ and $\nu = \lambda^a$;
similarly, if $\delta_b$ is a $\pi_{\mu, \nu}$-derivation then
$\mu = \lambda^{-b}$ and $\nu = 1$.
\end{remark}

\begin{lemma}\label{lem:nct}
With $\alg_0$ as in Definition~\ref{dfn:nct}, and ${}_a\delta$ and
$\delta_b$ as in Lemma~\ref{lem:newderivs}, fix $c_1$,~$c_2 \in \C$
and let
\[
\phi : \alg_0 \to \alg_0 \otimes \bop{\C^3}{}; \ %
x \mapsto \begin{bmatrix}
\tau( x ) & \overline{c_1} \, {}_a\delta^\dagger( x )
 & \overline{c_2} \, \delta^\dagger_b( x ) \\[1ex]
 c_1 \, {}_a\delta( x ) & \pi_{1, \lambda^a}( x ) - x & 0 \\[1ex]
 c_2 \, \delta_b( x ) & 0 & \pi_{\lambda^{-b}, 1}( x ) - x
\end{bmatrix},
\]
where the map
\[
\tau : %
\alg_0 \to \alg_0; \ %
U^m V^n \mapsto %
 -\hlf \bigl( | c_1 |^2 m^2 + | c_2 |^2 n^2 \bigr) U^m V^n.
\]
Then $\tau$ is $*$-linear and $\phi$ is a flow generator.
\end{lemma}

\begin{lemma}\label{lem:inorout}
Let $\phi$ be as in Lemma~\ref{lem:nct}. If $a = b = 0$ then
$\alg_\phi = \alg_0$; conversely, if $a \neq 0$ and $c_1 \neq 0$ then
$U \notin \alg_\phi$, and if $b \neq 0$ and $c_2 \neq 0$ then
$V \notin \alg_\phi$.
\end{lemma}
\begin{proof}
When $a = b = 0$, note that $\phi( U ) = U \otimes m_U$ and
$\phi( V ) = V \otimes m_V$, where
\[
m_U := \begin{bmatrix} -\hlf | c_1 |^2 & -\overline{c_1} & 0 \\[1ex]
c_1 & 0 & 0 \\[1ex] 0 & 0 & 0 \end{bmatrix} %
\qquad \text{and} \qquad %
m_V := \begin{bmatrix} -\hlf | c_2 |^2 & 0 & -\overline{c_2} \\[1ex]
0 & 0 & 0 \\[1ex] c_2 & 0 & 0 \end{bmatrix}.
\]
Hence $\phi_n( U ) = U \otimes m_U^{\otimes n}$ and
$\phi_n( V ) = V \otimes m_V^{\otimes n}$, so $U$, $V \in \alg_\phi$,
as claimed, and $\alg_\phi = \alg_0$, by Corollary~\ref{cor:bound}.

If $a > 0$ then, by induction, one gets that
\[
{}_a\delta^n( U ) = %
\prod_{i = 0}^{n - 1} \bigl( i a + 1 \bigr) U^{a n + 1} %
\qquad \text{for all } n \in \N.
\]
Let $e = [1 \ 0 \ 0]^T$ and $f = [0 \ 1 \ 0]^T$ be unit vectors
in~$\C^3$, and note that
\[
\bigl( \id_\ini \otimes \bra{f} \otimes \cdots %
\otimes \bra{f} \bigr) \phi_n( x ) %
\bigl( \id_\ini \otimes \ket{e} \otimes \cdots %
\otimes \ket{e} \bigr) = c_1^n \, {}_a\delta^n( x ) %
\qquad \text{for all } x \in \alg_0,
\]
so
\[
\| \phi_n( U ) \| \ge %
| c_1 |^n \prod_{i = 0}^{n - 1} \bigl( i a + 1 \bigr) \ge %
| c_1 |^n n!.
\]
If $a < 0$ then, by considering ${}_a\delta^\dagger$ instead, we see
that
\[
\| \phi_n( U ) \| \ge \| \bigl( %
\id_\ini \otimes \bra{e} \otimes \cdots \otimes \bra{e} \bigr) %
\phi_n( U ) \bigl( %
\id_\ini \otimes \ket{f} \otimes \cdots \otimes \ket{f} \bigr) \| %
\ge | c_1 |^n n!.
\]
A similar proof shows that $V \notin \alg_\phi$ when $b \neq 0$.
\end{proof}

\begin{remark}
The lower bounds obtained in Lemma~\ref{lem:inorout} when $a \neq 0$
or $b \neq 0$ show that our techniques do not apply in these
cases. The same problem arises if one attempts to use the results of
\cite{FaS93} instead.
\end{remark}

The following theorem gives the existence of a quantum flow used by
Goswami, Sahu and Sinha \cite[Theorem~2.1(i)]{GSS05}.

\begin{theorem}\label{thm:nctcgs}
Let $\alg$ be as in Definition~\ref{dfn:nct} and $\phi$ as in
Lemma~\ref{lem:nct} for $a = b = 0$. There exists an adapted family
$j$ of unital $*$-homomorphisms from $\alg$ to
$\bop{\ini \ctimes \fock}{}$ such that
\[
\langle u \evec{f}, j_t( x ) v \evec{g} \rangle = %
\langle u \evec{f}, ( x v ) \evec{g} \rangle + %
\int_0^t \langle u \evec{f}, j_s\bigl( %
\phi^{\wh{f}( s )}_{\wh{g}( s )}( x ) \bigr) v \evec{g} %
\rangle \std s
\]
for all $u$,~$v \in \ini$, $f$,~$g \in \elltwo$, $x \in \alg_0$ and
$t \ge 0$.
\end{theorem}
\begin{proof}
This follows from Theorem~\ref{thm:extension2}, Lemma~\ref{lem:nct}
and Lemma~\ref{lem:inorout}.
\end{proof}

\begin{remark}
The cocycle constructed in Theorem~\ref{thm:nctcgs} is essentially a
classical object: as noted in \cite[Theorem~2.1]{CGS03}, when
$c_1 = c_2 = \I$ one may take
\[
j_t( x ) :=
\beta\bigl( \exp( 2 \pi \I B^1_t ),
 \exp( 2 \pi \I B^2_t ) \bigr)( x )
\qquad \text{for all } x \in \alg \text{ and } t \ge 0,
\]
where $\beta : \T^2 \to \Aut( \alg )$ is the natural action of the
$2$-torus $\T^2$ on $\alg$, so that
\[
\beta( z, w )( U^m V^n ) = z^m w^n U^m V^n
\qquad \text{for all } ( z, w ) \in \T^2,
\]
and the Fock space $\fock$ is identified in the usual manner with the
$L^2$~space of the two-dimensional classical Brownian motion
$( B^1, B^2 )$.
\end{remark}

The existence of flows where the generator has non-zero gauge
part may also be established.

\begin{lemma}\label{lem:gauge}
Fix $( \mu, \nu ) \in \T^2$ with $\mu \neq 1$. Let $\alg_0$ be as in
Definition~\ref{dfn:nct} and $\pi_{\mu,\nu}$ as in
Definition~\ref{dfn:rotates}. There exists a flow generator
\[
\phi : \alg_0 \to \alg_0 \otimes \bop{\C^2}{}; \ %
x \mapsto \begin{bmatrix}
\tau( x ) & -\mu \delta( x ) \\[1ex]
\delta( x ) & \pi_{\mu, \nu}( x ) - x
\end{bmatrix},
\]
where the $\pi_{\mu, \nu}$-derivation
\begin{equation}\label{eqn:pidact}
\delta : \alg_0 \to \alg_0; \ U^m V^n \mapsto %
\frac{1 - \mu^m \nu^n}{1 - \mu} U^m V^n
\end{equation}
is such that $\delta^\dagger = -\mu \delta$, and the map
\[
\tau := \frac{\mu}{1 - \mu} \delta : \alg_0 \to \alg_0; \ %
U^m V^n \mapsto %
\frac{\mu ( 1 - \mu^m \nu^n )}{( 1 - \mu )^2} U^m V^n.
\]
Furthermore, $U$, $V \in \alg_\phi$ and so $\alg_\phi = \alg_0$.
\end{lemma}
\begin{proof}
Using the basis $\{ U^m V^n: m, n \in \Z \}$, one can readily verify
that $\delta$ is a $\pi_{\mu, \nu}$-derivation such that
$\delta^\dagger = -\mu \delta$, and hence $\phi$ is a flow
generator. Since
\[
\phi( U ) = U \otimes %
\begin{bmatrix}
 \ds\frac{\mu}{1 - \mu} & -\mu \\[2ex]
 1 & \mu - 1
\end{bmatrix} \qquad \text{and} \qquad %
\phi( V ) = V \otimes \frac{1 - \nu}{1 - \mu}%
\begin{bmatrix}
 \ds\frac{\mu}{1 - \mu} & -\mu \\[2ex]
 1 & \mu - 1
\end{bmatrix},
\]
the fact that $\{ U, V \} \subseteq \alg_\phi$ follows as in the proof
of Lemma~\ref{lem:inorout}.
\end{proof}

\begin{remark}
It is curious that for $\phi$ as in Lemma~\ref{lem:gauge} we have
$\tau = \mu ( 1 - \mu )^{-1} \delta$, and so $\tau$ is first rather
than second order. Whether or not $\phi$ or, equivalently, $\delta$ is
bounded is an open question; our existence result obviates the need to
determine this.
\end{remark}

\begin{theorem}
Let $\alg$ be as in Definition~\ref{dfn:nct} and $\phi$ as in
Lemma~\ref{lem:gauge}. There exists an adapted family $j$ of unital
$*$-homomorphisms from $\alg$ to $\bop{\ini \ctimes \fock}{}$ such
that
\[
\langle u \evec{f}, j_t( x ) v \evec{g} \rangle = %
\langle u \evec{f}, ( x v ) \evec{g} \rangle + %
\int_0^t \langle u \evec{f}, j_s\bigl( %
\phi^{\wh{f}( s )}_{\wh{g}( s )}( x ) \bigr) v \evec{g} %
\rangle \std s
\]
for all $u$,~$v \in \ini$, $f$,~$g \in \elltwo$, $x \in \alg_0$ and
$t \ge 0$.
\end{theorem}

As noted by Hudson and Robinson~\cite{HuR88}, the following
result makes clear why in Theorem~\ref{thm:nctcgs} it is necessary to
use two dimensions of noise to obtain a process whose flow generator
includes both of the derivations
$c_1 \, {}_0\delta$ and $c_2 \, \delta_0$: the linear combination
$\delta = c_1 \, {}_0\delta + c_2 \, \delta_0$ can appear on the
right-hand side of \eqref{eqn:tau} only when the coefficients $c_1$
and $c_2$ satisfy a particular algebraic relation.

\begin{proposition}\label{prp:nogo}
Let ${}_0\delta$ and $\delta_0$ be as in Lemma~\ref{lem:newderivs},
and let $\delta = c_1 \, {}_0 \delta + c_2 \, \delta_0$ for complex
numbers $c_1$ and $c_2$. A necessary and sufficient condition for the
existence of a linear map $\tau : \alg_0 \to \alg$ such that
\[
\tau( x y ) - \tau( x ) y - x \tau( y ) = %
\delta^\dagger( x ) \delta( y ) \qquad \text{for all } x, y \in \alg_0
\]
is the equality $c_1 \overline{c_2} = \overline{c_1} c_2$.
\end{proposition}
\begin{proof}
This may be established by adapting slightly the proof of
\cite[Theorem~2.2]{Rob90}.
\end{proof}

\subsection{The universal rotation algebra}

To avoid the issue of Proposition~\ref{prp:nogo}, Hudson and
Robinson work with the universal rotation algebra.

\begin{definition}\label{dfn:ura}
Let $\alg$ be the \emph{universal rotation algebra} \cite{AnP89}: this
is the universal $C^*$~algebra with unitary generators $U$, $V$ and
$Z$ satisfying the relations
\[
U V = Z V U, \qquad U Z = Z U \quad \text{and} \quad V Z = Z V.
\]
It may be viewed as the group $C^*$~algebra corresponding to the
discrete Heisenberg group
\[
\Gamma := %
\langle u, v, z \mid u v = z v u, \ u z = z u, \ v z = z v \rangle;
\]
from this perspective, its universal nature is immediately apparent.

Letting $\alg_0$ denote the $*$-subalgebra generated by $U$, $V$
and~$Z$, there are skew-adjoint derivations
\[
\delta_1 : \alg_0 \to \alg_0; \ U^m V^n Z^p \mapsto m U^m V^n Z^p %
\quad \text{and} \quad %
\delta_2 : \alg_0 \to \alg_0; \ U^m V^n Z^p \mapsto n U^m V^n Z^p
\]
for all $m$, $n$, $p \in \Z$.
\end{definition}

\begin{remark}
For a concrete version of the universal rotation algebra, let
$\ini := \ell^2( \Z^3 )$ and define operators $U_c$, $V_c$ and $Z_c$
by setting
\[
( U_c u )_{m, n, p} = u_{m + 1, n, p}, \quad %
( V_c u )_{m, n, p} = u_{m, n + 1, m + p} \quad \text{and} \quad %
( Z_c u )_{m, n, p} = u_{m, n, p + 1}
\]
for all $u \in \ini$ and $m$, $n$, $p \in \Z$. It is readily verified
that $U_c$, $V_c$ and $Z_c$ are unitary and satisfy the commutation
relations as claimed; let $\alg_c$ be the $C^*$~algebra generated by
these operators.

Universality gives a surjective $*$-homomorphism from $\alg$ to
$\alg_c$ such that $U \mapsto U_c$, $V \mapsto V_c$ and
$Z \mapsto Z_c$, and injectivity may be established in the same manner
as for the non-commutative torus: there is a faithful state $\tau$ on
$\alg$ such that $\tau( U^m V^n Z^p ) = \tfn{m = n = p = 0}$ and this
corresponds to the vector state given by $e \in \ini$ such that
$e_{m, n, p} = \tfn{m = n = p = 0}$.
\end{remark}

\begin{lemma}\label{lem:ura}
With $\alg_0$, $\delta_1$ and $\delta_2$ as in
Definition~\ref{dfn:ura}, fix $c_1$,~$c_2 \in \C$, let
$\delta = c_1 \delta_1 + c_2 \delta_2$ and define the
\emph{Bellissard map}
\[
\tau : \alg_0 \to \alg_0; \ U^m V^n Z^p \mapsto %
- \bigl( \hlf | c_1 |^2 m^2 + \hlf | c_2 |^2 n^2 + %
\overline{c_1} c_2 m n + %
( \overline{c_1} c_2 - c_1 \overline{c_2} ) p \bigr) U^m V^n Z^p,
\]
Then $\tau$ is $*$-linear and such that
\[
\tau( x y ) - \tau( x ) y - x \tau( y ) = %
\delta^\dagger( x ) \delta( y ) %
\qquad \text{for all } x, y \in \alg_0,
\]
so the map
\[
\phi : \alg_0 \to \alg_0 \otimes \bop{\C^2}{}; \ %
x \mapsto \begin{bmatrix} \tau( x ) & \delta^\dagger( x ) \\[1ex]
\delta( x ) & 0 \end{bmatrix}
\]
is a flow generator.

Furthermore, $U$, $V$, $Z \in \alg_\phi$ and $\alg_\phi = \alg_0$.
\end{lemma}
\begin{proof}
The algebraic statements are readily verified, and a short calculation
shows that
\[
\phi( U ) = U \otimes m_U, \qquad %
\phi( V ) = V \otimes m_V \qquad \text{and} \qquad %
\phi( Z ) = Z \otimes m_Z,
\]
where
\[
m_U = \begin{bmatrix}
-\hlf | c_1 |^2 & -\overline{c_1} \\[1ex]
 c_1 & 0
\end{bmatrix}, \qquad %
m_V = \begin{bmatrix}
-\hlf | c_2 |^2 & -\overline{c_2} \\[1ex]
 c_2 & 0
\end{bmatrix} \qquad \text{and} \qquad %
m_Z = \begin{bmatrix}
 c_1 \overline{c_2} - \overline{c_1} c_2 & 0 \\[1ex]
 0 & 0
\end{bmatrix}.
\]
Hence
\[
\phi_n( U ) = U \otimes m_U^{\otimes n}, \qquad %
\phi_n( V ) = V \otimes m_V^{\otimes n} \qquad \text{and} \qquad %
\phi_n( Z ) = Z \otimes m_Z^{\otimes n}
\]
for all $n \in \Z_+$, so $U$, $V$, $Z \in \alg_\phi$ and
$\alg_\phi = \alg_0$, by Corollary~\ref{cor:bound}.
\end{proof}

The following theorem is an algebraic version of the result
presented by Hudson and Robinson in \cite[Section~4]{HuR88}.

\begin{theorem}
Let $\alg$ be as in Definition~\ref{dfn:ura} and $\phi$ as in
Lemma~\ref{lem:ura}. There exists an adapted family $j$ of unital
$*$-homomorphisms from $\alg$ to $\bop{\ini \ctimes \fock}{}$ such
that
\[
\langle u \evec{f}, j_t( x ) v \evec{g} \rangle = %
\langle u \evec{f}, ( x v ) \evec{g} \rangle + %
\int_0^t \langle u \evec{f}, j_s\bigl( %
\phi^{\wh{f}( s )}_{\wh{g}( s )}( x ) \bigr) v \evec{g} %
\rangle \std s
\]
for all $u$,~$v \in \ini$, $f$,~$g \in \elltwo$, $x \in \alg_0$ and
$t \ge 0$.
\hfill$\Box$
\end{theorem}

\section*{Acknowledgements}
ACRB thanks Professors Kalyan Sinha and Tirthankar Bhattacharyya for
hospitality at the Indian Institute of Science, Bangalore, and in
Munnar, Kerala; part of this work was completed during a visit to
India supported by the UKIERI research network \emph{Quantum
Probability, Noncommutative Geometry and Quantum Information}. Thanks
are also due to Professor Martin Lindsay for helpful discussions.
Funding from Lancaster University's Research Support Office and
Faculty of Science and Technology is gratefully acknowledged.

SJW thanks Professor Rolando Rebolledo for a very pleasant visit to
Santiago in 2006 where thoughts about the quantum exclusion process
were first encouraged.

Both authors are indebted to the two anonymous referees and the
associate editor for their constructive comments on an earlier draft
of this paper.

\end{document}